\documentclass[10pt]{amsart}

\usepackage[T1]{fontenc}
\usepackage[utf8]{inputenc}
\usepackage{amsmath,amssymb,mathtools,stmaryrd,thmtools}
\usepackage{graphicx}
\usepackage{xcolor}
\usepackage{fancyhdr}
\usepackage{hyperref}
\hypersetup{colorlinks=true, linkcolor=blue, citecolor=magenta, urlcolor=green}
\usepackage[noabbrev]{cleveref}
\usepackage{comment}
\usepackage{url}
\usepackage{tikz-cd}
\usepackage{xy}
\input xy
\xyoption{all}
\usepackage{calrsfs}

\makeatletter
\def\thm@space@setup{%
  \thm@preskip=6pt
  \thm@postskip=6pt
}
\makeatother

\newtheorem{theorem}{Theorem}[section]
\newtheorem{lemma}[theorem]{Lemma}

\newtheorem{corollary}[theorem]{Corollary}

\theoremstyle{definition}

\newtheorem{example}[theorem]{Example}

\newtheorem{conjecture}[theorem]{Conjecture}

\theoremstyle{remark}



\crefname{theorem}{Theorem}{Theorems}
\Crefname{theorem}{Theorem}{Theorems}
\crefname{lemma}{Lemma}{Lemmas}
\Crefname{lemma}{Lemma}{Lemmas}
\crefname{proposition}{Proposition}{Propositions}
\Crefname{proposition}{Proposition}{Propositions}
\crefname{corollary}{Corollary}{Corollaries}
\Crefname{corollary}{Corollary}{Corollaries}
\crefname{definition}{Definition}{Definitions}
\Crefname{definition}{Definition}{Definitions}
\crefname{example}{Example}{Examples}
\Crefname{example}{Example}{Examples}
\crefname{conjecture}{Conjecture}{Conjectures}
\Crefname{conjecture}{Conjecture}{Conjectures}
\crefname{equation}{Equation}{Equations}
\Crefname{equation}{Equation}{Equations}

\numberwithin{equation}{section}

\newcommand{\cc}{\mathbb{C}}

\newcommand{\nn}{\mathbb{N}}

\newcommand{\qq}{\mathbb{Q}}
\newcommand{\Q}{\mathbb{Q}}
\newcommand{\rr}{\mathbb{R}}
\newcommand{\zz}{\mathbb{Z}}

\newcommand{\red}{\mathrm{red}}
\newcommand{\atoms}{\mathcal{A}}

\newcommand{\Z}{\mathsf{Z}}
\newcommand{\Lset}{\mathsf{L}}

\newcommand{\supp}{\operatorname{supp}}


\title{The Bi-UFS Positive Conjecture for Algebraic Semidomains}

\author{Aaditya Bilakanti}
\address{}
\email{abilak@gmail.com}

\author{Marly Gotti}
\address{iMath-Lab\\Cambridge, CA 02139}
\email{marlygotti@imath-lab.com}

\author{Amrit Kandasamy}
\address{}
\email{akandas5@asu.edu}

\author{Hengrui Liang}
\address{}
\email{spaceblastxy1@gmail.com}

\author{Jonathan Liu}
\address{}
\email{jonathanliu2024@gmail.com}

\author{Harold Polo}
\address{iMath-Lab\\Cambridge, CA 02139}
\email{harry.polo@imath-lab.com}

\author{Jason Yang}
\address{}
\email{jyang.super@gmail.com}

\author{Alan Yao}
\address{}
\email{alanyao2008@gmail.com}

\date{\today}

\begin{document}

\begin{abstract}
A semidomain is called \emph{bi-UFS} if both its additive monoid and its nonzero multiplicative monoid are unique factorization monoids. The \emph{Bi-UFS Positive Conjecture} predicts that the only positive semidomain with this property is $\nn_0$. We prove this conjecture for finitely generated algebraic positive semidomains. In the cyclic case, we show that for every positive algebraic number $\alpha$, the semidomain $\nn_0[\alpha]$ is bi-UFS if and only if $\alpha \in \nn$, equivalently $\nn_0[\alpha]=\nn_0$. The proof separates the quadratic case, where an analysis of the least additive atom larger than $1$ leaves only the examples $\nn_0[\sqrt 2]$ and $\nn_0[(1+\sqrt 5)/2]$ to exclude, from the higher-degree case, where explicit multiplicative identities force the minimal polynomial into impossible forms. We then give a Perron--Frobenius argument showing that if $\alpha_1,\ldots,\alpha_n$ are positive algebraic numbers and $\nn_0[\alpha_1,\ldots,\alpha_n]$ is bi-UFS then this semidomain is $\nn_0$. Finally we prove a reduction theorem for complex semidomains: every bi-UFS subsemidomain of $\cc$ that is not a UFD and has finitely many additive atoms is isomorphic to a positive semidomain. Consequently, every finitely generated algebraic bi-UFS semidomain over $\cc$ is either an integral domain or isomorphic to $\nn_0$.
\end{abstract}

\maketitle

\bigskip

\section{Introduction}
\label{sec:intro}

The atomic structure of positive monoids, that is, additive submonoids of \(\rr_{\ge 0}\), has been studied extensively. Numerical monoids give the simplest examples, and Puiseux monoids, namely additive submonoids of \(\qq_{\ge 0}\), have also been examined in detail; see~\cite{CGG20,CGGP26} and the references therein. A subset \(S\) of the complex field \(\cc\) is a \emph{complex semiring} if it contains \(0\) and \(1\) and is closed under the usual addition and multiplication. A complex semiring consisting of positive real numbers is called positive, and a complex semiring consisting of algebraic numbers is called algebraic. For a complex semiring \(S\), both the additive monoid \((S,+)\) and the multiplicative monoid \((S^*,\cdot)\), where \(S^*:=S\setminus\{0\}\), are natural objects of study. Their interaction is the central theme of this paper.

A cancellative commutative monoid is a \emph{unique factorization monoid} (UFM) if every nonunit factors uniquely into atoms. A cancellative commutative monoid is \emph{atomic} if every nonunit factors into finitely many atoms, and an atomic monoid is a \emph{half-factorial monoid} (HFM) if every two factorizations of the same element have the same number of atoms, counted with multiplicity. A complex semidomain is called a \emph{bi-UFS} (resp., \emph{bi-HFS}) if both its additive and multiplicative monoids are UFMs (resp., HFMs).

The prototypical semidomain \(\nn_0\) is a bi-UFS: its additive monoid is the free monoid on \(\{1\}\), and its multiplicative monoid is the free multiplicative monoid on the rational primes. In their study of bi-atomic semirings, Baeth, Chapman, and Gotti~\cite{BCG21} proposed the following conjecture, which is one of the main motivations for this paper.

\begin{conjecture}[Bi-UFS Positive Conjecture \protect{\cite[Conjecture~7.7]{BCG21}}]\label{conj:biUF}
A positive semidomain \(S\) is a bi-UFS if and only if \(S=\nn_0\).
\end{conjecture}

The conjecture is natural because the two factorization structures pull in opposite directions. Additive unique factorization makes a positive semiring resemble a free additive monoid, whereas multiplicative unique factorization imposes strong divisibility restrictions on the same elements. The main purpose of this paper is to give partial positive answers to the Bi-UFS Positive Conjecture.

The phenomenon of nonunique factorization has been widely investigated in commutative monoids and, more recently, in semidomains. Interest in the atomic structure, ideal theory, and factorization properties of monogenic semidomains increased after the appearance of the papers~\cite{CF19} by Campanini and Facchini and~\cite{CGG20} by Chapman, Gotti, and Gotti. The paper~\cite{CF19} studies factorization and ideal-theoretic aspects of the transcendental monogenic semidomain \(\nn_0[x]\), while~\cite{CGG20} investigates factorization and sets of lengths in the additive structure of monogenic rational semidomains. For further studies of additive monoids of monogenic semidomains, see~\cite{ABP21,hP23}; for algebraic generators, see~\cite{ABLST23,DDGLPVZ25}. Rational examples had appeared earlier in~\cite[Section~5]{GG18}. Recent work on multiplicative monoids of monogenic semidomains includes~\cite{DGZ26}.

The main class of complex semidomains considered in this paper consists of monogenic semidomains, namely simple semiring extensions
\[
  \nn_0[\rho]=\{f(\rho): f(x)\in \nn_0[x]\}
\]
of \(\nn_0\) by an element \(\rho\in\cc\). To simplify notation, set \(S_\rho:=\nn_0[\rho]\), and call \(S_\rho\) rational, positive, real, or algebraic according as \(\rho\) has the corresponding property.

Throughout this paper, the following characterization of additive factoriality and additive half-factoriality for algebraic monogenic semidomains, due to Correa-Morris and Gotti, plays an essential role.

\begin{theorem}[Correa-Morris--Gotti, \protect{\cite[Theorem~5.4]{CG22}}]\label{prop:additively-factoriality-characterizations}
Let \(\alpha\) be a positive algebraic number with minimal polynomial \(m_\alpha(x)\in\qq[x]\). Then the following conditions are equivalent.
\begin{enumerate}
  \item[(a)] \((S_\alpha,+)\) is a UFM.
  \smallskip
  \item[(b)] \((S_\alpha,+)\) is an HFM.
  \smallskip
  \item[(c)] \(\deg m_\alpha(x)=|\atoms^+(S_\alpha)|\).
\end{enumerate}
\end{theorem}

This paper has two aims. First, we resolve~\Cref{conj:biUF} for several natural algebraic classes of positive semidomains, beginning with algebraic monogenic semidomains and then extending this result to semirings generated by finitely many positive algebraic numbers. Second, we connect the positive case to the complex algebraic setting through a reduction theorem.

For an algebraic number \(\alpha\in\rr_{>0}\), our first main result states that the monogenic semidomain \(\nn_0[\alpha]\) is a bi-UFS if and only if \(\alpha\in\nn\), in which case \(\nn_0[\alpha]=\nn_0\). The proof divides according to the algebraic degree of \(\alpha\). Although the degree-two case was previously established by Gotti, Graia, Han, and Liang~\cite[Theorem~3.6]{GGHL26},~\Cref{sec:deg2} gives a different approach through an analysis of the smallest nontrivial additive atom
\[
  \ell:=\min(\atoms^+(S)\setminus\{1\}),
\]
where \(S=\nn_0[\alpha]\). We show, by explicit algebraic identities and a sequence of interval exclusions, that \(\ell\) must satisfy \(1<\ell\le \phi\), where \(\phi=(1+\sqrt 5)/2\) is the golden ratio. The remaining minimal polynomials are then ruled out via explicit nonunique factorizations in \(\nn_0[\sqrt 2]\) and \(\nn_0[\phi]\). The general degree case, treated in~\Cref{sec:deg3up}, follows a different route. The parametric identity
\[
    (1 + \alpha^r + \alpha^{2r})(1 + \alpha^{3r}) = (1 + \alpha^r)(1 + \alpha^{2r} + \alpha^{4r})
\]
extracts structural information about the minimal polynomial of \(\alpha\) from the failure of certain elements to be multiplicative atoms. This reduces the minimal polynomial to one of two specific forms, each of which is then ruled out directly.

Our second main result, established in~\Cref{sec:multigenerator}, extends the cyclic theorem to finitely many positive algebraic generators. If \(\alpha_1,\ldots,\alpha_n\in\rr_{>0}\) are algebraic and \(S:=\nn_0[\alpha_1,\ldots,\alpha_n]\) is a bi-UFS then \(S=\nn_0\). The argument is geometric rather than combinatorial. The additive atoms of \(S\) generate a rational cone \(C\) inside the number field \(K:=\qq(\alpha_1,\ldots,\alpha_n)\), and a Perron--Frobenius analysis of the matrix representing multiplication by an interior element of \(C\) shows that \(C\) is closed under inversion. An extreme-ray argument then forces \(C\) to be one-dimensional, and so \(S=\nn_0\). The cyclic case follows as a special case of this theorem, although the independent treatment of the cyclic case is more elementary and yields the structural bound \(\ell\le\phi\).

Finally, in~\Cref{sec:complex}, we move beyond the positive setting and consider bi-UFS semidomains \(S\subseteq\cc\) with finitely many additive atoms. We show that such an \(S\) is isomorphic to a positive semidomain, with the isomorphism built from the Perron eigenvector of a multiplication matrix. Combined with the multigenerator result and with the observation that every finitely generated algebraic semidomain (that is not a UFD) over \(\cc\) has finitely many additive atoms, this reduction shows that the finitely generated algebraic complex case follows from the positive case.

\Cref{sec:prelim} collects the notation and background from factorization theory and nonnegative matrices used throughout.

\bigskip

\section{Preliminaries}\label{sec:prelim}

We collect the notation and background used throughout. For commutative monoids and factorization theory we follow~\cite{GH06}, and for nonnegative matrices we use~\cite{HJ13}.

\subsection{General Notation}

We write $\nn$ for the positive integers and $\nn_0 := \nn \cup \{0\}$. For $a, b \in \zz$, set $\llbracket a,b \rrbracket := \{n \in \zz : a \le n \le b\}$. For $X \subseteq \rr$ and $r \in \rr$, write $X_{\ge r} := \{x \in X : x \ge r\}$, and similarly for $X_{>r}$, $X_{\le r}$, $X_{<r}$. The support of $f(x) \in \qq[x]$, denoted $\supp(f(x))$, is the set of exponents of nonzero monomials in $f$.

\subsection{Monoids and Factorizations}

Throughout the paper, the term \emph{monoid} means a commutative, cancellative semigroup with identity. Let $M$ be a monoid with identity $\iota$ and set $M^\bullet := M \setminus \{\iota\}$. We let $\mathcal U(M)$ denote the group of units of $M$ and call $M$ \emph{reduced} if $\mathcal U(M) = \{\iota\}$. For $x, y \in M$, we write $x \mid_M y$ if $y = xz$ for some $z \in M$, and reserve $\mid$ for divisibility in $\zz$ or $\nn$.

An element $a \in M \setminus \mathcal U(M)$ is an \emph{atom} if $a = xy$ in $M$ forces $x \in \mathcal U(M)$ or $y \in \mathcal U(M)$, and we denote the set of atoms by $\atoms(M)$. The monoid $M$ is \emph{atomic} provided every element of $M \setminus \mathcal U(M)$ is a product of atoms. An ascending chain of principal ideals $(x_n M)_{n \in \nn}$ is said to \emph{stabilize} if $x_n M = x_{n+1} M$ for all sufficiently large $n$. The monoid $M$ satisfies the \emph{ascending chain condition on principal ideals} (ACCP) if every such chain stabilizes. By~\cite[Proposition~1.1.4]{GH06}, every monoid satisfying the ACCP is atomic.

A \emph{free commutative monoid} on a set $P$ is a commutative monoid $F$ together with a map $\iota \colon P \to F$ such that every map from $P$ to a commutative monoid $N$ extends uniquely to a monoid homomorphism $F \to N$. Concretely, $F$ may be realized as the set of formal finite products $\prod_{p \in P} p^{n_p}$ with $n_p \in \nn_0$ and $n_p = 0$ for all but finitely many $p$, multiplied componentwise. For each set $P$, the free commutative monoid on $P$ is unique up to isomorphism.

Suppose $M$ is atomic, and let $M_{\red} := M / \mathcal U(M)$ denote its reduced monoid. Write $\Z(M)$ for the free commutative monoid on $\atoms(M_{\red})$, and let $\pi \colon \Z(M) \to M_{\red}$ be the unique monoid homomorphism fixing $\atoms(M_{\red})$ elementwise. For each $b \in M$, the \emph{set of factorizations} of $b$ is $\Z_M(b) := \pi^{-1}\bigl(b\,\mathcal U(M)\bigr)$, and the \emph{set of lengths} of $b$ is $\Lset_M(b) := \{\,|z| : z \in \Z_M(b)\,\}$, where $|z|$ denotes the length of $z$ as a word in $\Z(M)$, that is, the total exponent $\sum_{p} n_p$ in the formal product representation of $z$. We call $M$ a \emph{unique factorization monoid} (UFM) if $|\Z_M(b)| = 1$ for every $b \in M$, and a \emph{half-factorial monoid} (HFM) if $|\Lset_M(b)| = 1$ for every $b \in M$. Every UFM satisfies the ACCP. When $M$ is written additively and is a UFM, we abbreviate the unique length of the factorization of $x$ by $|x|$, writing $|x|_M$ when the ambient monoid is not clear from context.

We will repeatedly use the following standard fact: if $M$ is a UFM and $a, b, c \in M$ with $a \mid_M bc$ and $\gcd_M(a, b) = \mathcal U(M)$ then $a \mid_M c$. This is immediate from the fact that atoms are prime in every UFM~\cite[Theorem~1.1.10]{GH06}.

\subsection{Positive Semidomains}

A \emph{semiring} is a triple $(S, +, \cdot)$ such that $(S, +)$ is a commutative monoid, $(S^\ast, \cdot)$ is a commutative semigroup, multiplication distributes over addition, and $0 \cdot x = x \cdot 0 = 0$ for every $x \in S$. A \emph{positive semidomain} is a subsemiring $S \subseteq \rr_{\ge 0}$ containing $1$. For such an $S$, the monoid $(S, +)$ is automatically reduced and $(S^\ast, \cdot)$ is a monoid. We write $\atoms^+(S)$ and $\atoms^\times(S)$ for the atom sets of $(S, +)$ and $(S^\ast, \cdot)$, respectively. The semidomain $S$ is \emph{bi-UFS} provided that both $(S, +)$ and $(S^\ast, \cdot)$ are UFMs. For $x, y \in S^\ast$, we write $x \mid_{S^\ast} y$ if there exists $m \in S^\ast$ with $xm = y$.

\subsection{Cyclic Algebraic Semidomains}

For $\alpha \in \rr_{>0}$, set
\[
  \nn_0[\alpha] := \{f(\alpha) : f(x) \in \nn_0[x]\}.
\]
When $\alpha$ is algebraic, write $m_\alpha(x) \in \qq[x]$ for its minimal polynomial. Let $\ell$ be the smallest positive integer with $\ell m_\alpha(x) \in \zz[x]$, and split the resulting integer polynomial into its positive and negative parts: $\ell m_\alpha(x) = p(x) - q(x)$, where $p(x), q(x) \in \nn_0[x]$ have disjoint supports. The pair $(p(x), q(x))$ is called the \emph{minimal pair} of $\alpha$. We will make extensive use of the following characterization.

\begin{theorem}[\protect{\cite[Theorem~5.4]{CG22}}]\label{thm:equivalent_conditions_UFM}
Let $\alpha \in \rr_{>0}$ be algebraic of degree $d$ with minimal polynomial $m_\alpha(x)$ and minimal pair $(p(x), q(x))$. Then the following conditions are equivalent.
\begin{enumerate}
  \item[(a)] $(\nn_0[\alpha], +)$ is a UFM.
  \smallskip
  \item[(b)] $(\nn_0[\alpha], +)$ is an HFM.
  \smallskip
  \item[(c)] $\deg m_\alpha(x) = |\atoms(\nn_0[\alpha])|$.
  \smallskip
  \item[(d)] $p(x) = x^d$.
\end{enumerate}
Moreover, when these equivalent conditions hold, $\atoms^+(\nn_0[\alpha]) = \{1, \alpha, \alpha^2, \ldots, \alpha^{d-1}\}$, and every element of $\nn_0[\alpha]$ has a unique representation as $f(\alpha)$ for some $f(x) \in \nn_0[x]$ of degree at most $d - 1$.
\end{theorem}

When $(\nn_0[\alpha], +)$ is a UFM, we call the unique $f(x) \in \nn_0[x]$ of degree at most $d - 1$ representing $a \in \nn_0[\alpha]$ the \emph{additive normal form} of $a$, and the sum of its coefficients the \emph{additive length} of $a$. Writing $m_\alpha(x) = x^d - q_\alpha(x)$,~\Cref{thm:equivalent_conditions_UFM} forces $q_\alpha(x) \in \nn_0[x]$, and $\alpha^d = q_\alpha(\alpha)$ governs all reductions to additive normal form.

\subsection{Nonnegative Matrices and Perron--Frobenius} A matrix $B \in M_r(\rr_{\ge 0})$ is \emph{irreducible} provided that there is no nonempty proper $I \subsetneq \{1, \ldots, r\}$ such that the coordinate subspace $\operatorname{span}_\qq \{e_i : i \in I\}$ is invariant under $B$. It is \emph{primitive} provided $B^k$ is entrywise positive for some $k \in \nn$. By~\cite[Lemma~8.5.4]{HJ13}, an irreducible nonnegative matrix with every positive diagonal entry is primitive.

\begin{theorem}[Perron--Frobenius, \protect{\cite[Theorems~8.4.4 and 8.5.1]{HJ13}}]\label{thm:PF}
Let $B \in M_r(\rr_{\ge 0})$ be irreducible. There exists $\rho(B) \in \rr_{>0}$, the \emph{Perron root} of $B$, with the following properties.
\begin{enumerate}
  \item $\rho(B)$ is a simple eigenvalue of $B$, and $|\lambda| \le \rho(B)$ for every eigenvalue $\lambda$ of $B$.
  \smallskip
  \item There is a strictly positive right eigenvector $w \in \rr^r_{>0}$ with $Bw = \rho(B)w$, and the $\rho(B)$-eigenspace is one-dimensional. The analogous statement holds for left eigenvectors.
  \smallskip
  \item If $B$ is primitive then $|\lambda| < \rho(B)$ for every eigenvalue $\lambda \ne \rho(B)$, and $\rho(B)^{-N} B^N \to (\lambda^\top w)^{-1} w \lambda^\top$ as $N \to \infty$, where $\lambda \in \rr^r_{>0}$ is a left Perron eigenvector. The limit matrix is entrywise positive.
\end{enumerate}
\end{theorem}

\bigskip

\section{Degree-Two Algebraic Monogenic Semidomains}\label{sec:deg2}

In this section, we show that, for a degree-two algebraic number \(\alpha\in\rr_{>0}\), the monogenic semidomain \(S = \nn_0[\alpha]\) is a bi-UFS if and only if \(\alpha\in\nn\). Although this case was previously established by Gotti, Graia, Han, and Liang~\cite[Theorem~3.6]{GGHL26}, we provide a different proof through an analysis of the smallest nontrivial additive atom
\[
  \ell:=\min(\atoms^+(S)\setminus\{1\}).
\]
The idea is to show, by explicit algebraic identities and a sequence of interval exclusions, that \(\ell\) must satisfy \(1<\ell\le \phi\), where \(\phi=(1+\sqrt 5)/2\) is the golden ratio. The remaining minimal polynomials are then ruled out via explicit non-unique factorizations in \(\nn_0[\sqrt 2]\) and \(\nn_0[\phi]\). 

We begin by establishing two lemmas about bi-UFS, which we will use thorought the paper.

\begin{lemma}\label{lem:one_atom}
  Let $S$ be a bi-UFS. Then the following statements hold.
    \begin{enumerate}
      \item $1 \in \atoms^+(S)$.
      \smallskip
      \item For all $x, y \in S$, the equality $|x + y| = |x| + |y|$ holds. In particular, $|x| = 0$ if and only if $x = 0$, and $|x| = 1$ if and only if $x \in \atoms^+(S)$.
      \smallskip
      \item Every multiplicative atom of $S^\ast$ is prime in $S^\ast$.
      \smallskip
      \item Every multiplicative divisor of a unit in $S^\ast$ is a unit in $S^\ast$.
    \end{enumerate}
\end{lemma}

\begin{proof}
    \textit{(1)} First, if $S=\nn_0$ then the claim is shown true. Otherwise, we can assume that there exists an atom $b\in\atoms^+(S)$ satisfying $b\neq1$. Let $1 = a_1 + a_2 + \cdots + a_n$ be a decomposition of $1$ into additive atoms. Then $b = ba_1 + ba_2 +\cdots + ba_n$. Since each of these factors are nonzero and $S$ is a reduced monoid with $0$ as its only additive unit, we must have $n=1$ to avoid a contradiction. Thus, $1$ must be an additive atom.
    \smallskip
    
    \textit{(2)} Concatenating the unique atomic factorization of $x$ and $y$ produces an atomic factorization of $x+y$ with length $|x| + |y|$, and uniqueness in $(S,+)$ forces this to be the only atomic factorization of $x+y$. Thus, $|x+y| = |x| + |y|$. The results about $|x| = 0$ and $|x| = 1$ follow from the definitions of $0$ and atoms.   
    \smallskip
    
    \textit{(3)} This is the standard fact that atoms are prime in every UFM; see~\cite[Theorem~1.1.10]{GH06}.
    \smallskip
    
    \textit{(4)} If $u \in \mathcal U(S^\ast)$ and $u = xy$ in $S^\ast$ then $1 = u u^{-1} = xy u^{-1} = x(yu^{-1})$, meaning that $yu^{-1}$ is the multiplicative inverse of $x$. Hence $x \in \mathcal U(S^\ast)$, and similarly $y \in \mathcal U(S^\ast)$.
\end{proof}

Next, we prove that, for a bi-UFS $S$, the inequality $|xy| \ge |x||y|$ holds for elements $x,y \in S$.

\begin{lemma}\label{lem:length_inequality}
  Let $S$ be a bi-UFS. For $x,y\in S$, the inequality $|xy| \ge |x||y|$ holds.
\end{lemma}

\begin{proof}
We may assume that $x, y \in S^*$. Now let \[x = a_1 + a_2 + \dots + a_m \;\hspace{.4 cm} \text{ and }\;\hspace{.4 cm} y = b_1 + b_2 + \dots + b_n\] be factorizations, which exist since $(S, +)$ is a UFM. Then $xy = \sum_{i,j} a_i b_j$ with each $a_i b_j > 0$, so each term's additive atomic decomposition contributes at least one atom into the factorization of $xy$. Hence $|xy| \ge mn = |x||y|$.
\end{proof}

\begin{corollary}\label{cor:divisor_of_atom}
    Let $S$ be a bi-UFS. If $x\in S$ is a multiplicative divisor of an atom $a \in \atoms^+(S)$ then $x\in\atoms^+(S)$. Moreover, every unit of $S^*$ is an element of $\atoms^+(S)$.
\end{corollary}

\begin{proof}
Since $a \in \atoms^+(S)$ and $a = xy$ for some $y\in S$, then $1 = |a| = |xy| \ge |x| |y|$ by~\Cref{lem:length_inequality}. Hence $|x| = |y| = 1$. Every unit in $S^*$ is in $\atoms^+(S)$ since it multiplicatively divides $1$, and $1 \in \atoms^+(S)$ by~\Cref{lem:one_atom}.
\end{proof}

\subsection[Absence of elements between 0 and 1]{Absence of elements in \texorpdfstring{$(0,1)$}{(0,1)}} The main result of this subsection is that a bi-UFS $S$ contains no element strictly between $0$ and $1$ (\Cref{lem:no_elements_between_zero_and_one}). We first rule out small additive atoms and then show that every element of $S$ is at least $1$. Throughout, let $d$ denote the unique positive solution of $d^3 + d = 1$, noting that $d \in (0,1)$. We first show that no multiplicative atom of $1 + u$ can divide $1 + u^3$.

\begin{lemma}\label{lem:atomic_divisor_1+u}
  Let $S$ be a bi-UFS. If $u \in \atoms^+(S)$ and $0 < u < 1$ then no atomic divisor of $1 + u$ divides $1 + u^3$.
\end{lemma}

\begin{proof}
  Let $c_1 = 1 + u$ and $c_2 = 1 + u^3$. Let $p$ be an atom in $S^*$ dividing $c_1$, so $c_1 = pa$. We first show that $a \in \atoms^+(S)$. \Cref{lem:length_inequality} gives the inequality $|c_1| = 2 = |pa| \ge |p||a|$. If $a$ is a multiplicative unit then it is already an atom in the additive monoid by~\Cref{cor:divisor_of_atom}. If $a$ is not a multiplicative unit then $|a| = 2$. In this case, we may let $a = b + c$ where $b$, $c \in \atoms^+(S)$. Then $1 + u = p(b + c) = pb + pc$. Since both $1$ and $u$ are additive atoms of $S$, one of $pb$ or $pc$ must equal $1$, which means that $p$ is a unit in $S^*$, contradicting the fact that $p$ is a multiplicative atom. Therefore $|a| = 1$, and $a \in \atoms^+(S)$.

  We now prove the main part of the lemma. Suppose that $p$ divides $c_2$. Let $c_2 = pb$. Since $c_2 < c_1$, it follows that $b < a$. Multiplying both sides of $c_2 = pb$ by $a$ yields $a c_2 = b c_1$. Expanding gives $a(1 + u^3) = b(1 + u)$, so we obtain $a + a u^3 = b + b u$. The left-hand side has the atom $a$ in the sum. However, $b < a$ and $bu < b < a$, so $a$ does not appear anywhere on the right-hand side. This contradicts unique factorization, so $p \nmid_{S^\ast} 1 + u^3$.
\end{proof}

From~\Cref{lem:atomic_divisor_1+u}, we prove a divisibility relation with a useful algebraic identity that will continue to be used throughout.

\begin{lemma}\label{lem:1+u_divides_1+u+u^2}
  Let $S$ be a bi-UFS. For every $u \in (0,1)$ with $u \in \atoms^+(S)$, the divisibility relation $1 + u \mid 1 + u + u^2$ holds.
\end{lemma}

\begin{proof}
  We use the identity
  \[
    (1 + u^3)(1 + u + u^2) = (1 + u)(1 + u^2 + u^4).
  \]
  Let $c_1 = 1 + u$,\hspace{.2 cm} $c_3 = 1 + u + u^2$, \hspace{.1 cm} and \hspace{.2 cm} $c_2 = 1 + u^3$. By~\Cref{lem:atomic_divisor_1+u}, no multiplicative atom dividing $c_1$ divides $c_2$, so because $S^*$ is atomic, we know $\gcd_{S^\ast}(c_1, c_2) = \mathcal U(S^*)$. Since the identity yields $c_1 \mid_{S^\ast} c_2 c_3$, it follows that $c_1$ divides $c_3$ as all UFMs have the D-property. Therefore $1 + u \mid_{S^\ast} 1 + u + u^2$. 
\end{proof}

When $u$ is small enough, \Cref{lem:1+u_divides_1+u+u^2} yields the stronger conclusion that $1 + u$ divides $u^2$. The threshold $d$ enters through the inequality $u + u^3 < 1$.

\begin{lemma}\label{lem:1+u_divides_u^2}
  Let $S$ be a bi-UFS positive semidomain. If $u \in \atoms^+(S)$ and $0 < u < d$ then $1 + u \mid_{S^\ast} u^2$.
\end{lemma}

\begin{proof}
Note that $d \in (0, 1)$. By~\Cref{lem:1+u_divides_1+u+u^2}, there is some $q \in S$ such that $1 + u + u^2 = (1 + u) q= q + u q$. Since $u < d$, it follows that $u + u^3 < 1$. Additionally,
  \[
    q = \frac{1 + u + u^2}{u + 1} = 1 + \frac{u^2}{1 + u} < 1 + u^2.
  \]
  Therefore $u q < u (1 + u^2) = u + u^3 < 1$.
  
  We examine the unique additive atomic decompositions of the left-hand and right-hand sides of the equality $1 + u + u^2 = q + u q$. The additive factorization of $u q$ will not include $1$ since $u q < 1$. Since the atom $u$ satisfies $u < 1$ and $u^2 < 1$, the left-hand side contains exactly one copy of the atom $1$, which must remain true for the atomic decomposition of the right-hand side, due to uniqueness. Therefore the unique copy of $1$ on the right-hand side must come from $q$, so $q = 1 + s$ for some nonzero $s \in S$. Substituting into $1 + u + u^2 = (1 + u) q$ yields
  \[
    1 + u + u^2 = (1 + u)(1 + s) = 1 + u + (1 + u) s,
  \]
  so $u^2 = s(1 + u)$. Therefore $1 + u \mid_{S^\ast} u^2$.
\end{proof}

We now rule out nontrivial multiplicative units. The
argument applies~\Cref{lem:1+u_divides_u^2} to a large power of a
hypothetical unit smaller than $1$.

\begin{lemma}\label{lem:multiplicative_monoid_reduced}
  Let $S$ be a bi-UFS. Then $S^\ast$ is reduced.
\end{lemma}

\begin{proof}
  Suppose for the sake of contradiction that $v \in S^\ast$ is a unit with $v \ne 1$. If $v > 1$, we may replace $v$ with $v^{-1}$, which would also be a unit, so we may assume without loss of generality that $0 < v < 1$. Since $v$ is a unit, so $v^n$ is also a unit for every positive integer $n$. By~\Cref{cor:divisor_of_atom}, every multiplicative unit is an atom in the additive monoid, so $v^n \in \atoms^+(S)$ for every $n \ge 1$. Choose $n$ large enough that $0 < v^n < d$, and set $u = v^n$. By~\Cref{lem:1+u_divides_u^2}, the element $1 + u$ divides $u^2$. However, $u^2=v^{2n}$ is a unit, and every divisor of a unit is a unit, so $1 + u$ is a unit and hence an additive atom, again by~\Cref{cor:divisor_of_atom}. However, $1$ is a proper additive divisor of $1+u$, a contradiction. Therefore the unit of $S^\ast$ is $1$.
\end{proof}

We now establish a lower bound on additive atoms by iterating~\Cref{lem:1+u_divides_u^2} along the multiplicative atoms of $1 + u$.

\begin{lemma}\label{lem:atoms_are_at_least_d}
  Let $S$ be a bi-UFS. Then $\atoms^+(S)\subseteq[d,\infty)$ for some $d \in \rr_{>0}$.
\end{lemma}

\begin{proof}
  Assume for the sake of contradiction that $u \in \atoms^+(S)$ and $0 < u < d$. By~\Cref{lem:1+u_divides_u^2}, $1 + u \mid_{S^\ast} u^2$. There are no nontrivial multiplicative units by~\Cref{lem:multiplicative_monoid_reduced}, so since $1+u\neq1$, we can factor $1 + u$ uniquely into multiplicative atoms in $1 + u = p_1 p_2 \dots p_k$. Since $1 + u > 1$, then there exists at least one $1\leq i\leq k$ such that $p_i>1$. Since $p_i$ divides $1 + u$ and $1 + u$ divides $u^2$, it follows that $p_i$ divides $u^2$. Since $S^*$ is a UFM, the atom $p_i$ must also be prime, so it follows that $p_i \mid_{S^\ast} u$. Thus $u = p_i u_1$ for some $u_1\in S$. Since $u \in \atoms^+(S)$, we know by~\Cref{cor:divisor_of_atom} that $u_1 \in \atoms^+(S)$. Since $p_i > 1$, it follows that $0 < u_1 < u < d$. Additionally, since $p_i$ is not a multiplicative unit, we know that $u_1$ is a proper divisor of $u$, or equivalently, $uS^*\subsetneq u_1S^*$.
  
  Since $u_1<d$, we know by~\Cref{lem:1+u_divides_u^2} that $1+u_1$ divides $u_1^2$. Because $1+u_1>1$, we may repeat the same argument with $u_1$ to obtain a $u_2 \in \atoms^+(S)$ such that $u_1S^*\subsetneq u_2S^*$ and $0 < u_2 < u_1 < d$. Continuing this yields a strictly ascending chain of principal ideals
  \[
    u S^* \subsetneq u_1 S^* \subsetneq u_2 S^* \subsetneq \cdots,
  \]
  which contradicts the ACCP in $S^*$.
\end{proof}

The main result of the subsection is now immediate.

\begin{lemma}\label{lem:no_elements_between_zero_and_one}
  Let $S$ be a bi-UFS. Then there are no elements between $0$ and $1$ in $S$.
\end{lemma}

\begin{proof}
  Assume for the sake of contradiction that $0 < \alpha < 1$ and $\alpha \in S$. We may choose $n$ large enough that $0 < \alpha^n < d$. Since $\alpha^n \in S$ and $(S, +)$ is atomic, there must exist an additive divisor of $\alpha^n$ that is an additive atom. Such an atom would be at most $\alpha^n$ and hence less than $d$, contradicting~\Cref{lem:atoms_are_at_least_d}. Therefore there is no $\alpha \in (0, 1)$ such that $\alpha \in S$.
\end{proof}

We now prove some useful results of the above statement.

\begin{corollary}\label{cor:more_facts}
  Let $S\subseteq \rr_{\ge 0}$ be a bi-UFS. If $x,y\in S^\ast$ and $x\mid_{S^\ast} y$ then $x\le y$. Additionally, for every $s\in S$, the inequality $|s|\le s$ holds.
\end{corollary}

\begin{proof}
By~\Cref{lem:no_elements_between_zero_and_one}, every element of
$S^\ast$ is at least $1$. Hence $x \mid_{S^\ast} y$ implies $x \le
y$, and since each additive atom of $S$ is at least $1$, the inequality $|s| \le s$ holds for every $s \in S$.
\end{proof}

\subsection{The Least Nontrivial Atom} Throughout this subsection we assume that $\ell := \min\bigl(\atoms^+(S) \setminus \{1\}\bigr)$ exists, and we work toward the bound $1 < \ell \le \phi$, where $\phi = (1 + \sqrt{5})/2$ denotes the golden ratio. We first record several structural facts about $\ell$ and the elements of $S$ below it, and then to rule out the intervals $(\phi, 2)$, $(2, 1 + \sqrt{2})$, and $[1 + \sqrt{2}, \infty)$ in turn.

We begin with an observation about divisibility by ordinary integers
in the presence of a nontrivial additive atom.

\begin{lemma}\label{lem:integers_do_not_divide_atom_plus_integer}
  Let $S$ be a bi-UFS. Let $a \in \atoms^+(S)$ with $a > 1$. For every integer $d \ge 2$ and every $n \in \nn_0$, the nondivisibility relation $d \nmid n + a$ holds.
\end{lemma}

\begin{proof}
  Suppose for the sake of contradiction that $n + a = d q$ for some $q\in S$. Since $d$ is an integer, we know \[n + a = \underbrace{q + q + \dots + q}_{d\text{ copies}}.\] Because of the uniqueness of additive atomic decompositions, every additive atom in the additive factorization of $d q$ has a multiplicity divisible by $d$. However, since every integer decomposes strictly into sums of the atom $1$, the atom $a$ in the additive factorization of $n + a$ appears with multiplicity $1$. Since $d \ge 2$, the multiplicity of $a$ is not divisible by $d$, so $d \nmid n + a$.
\end{proof}

Below $\ell$, every element of $S$ is an integer.

\begin{lemma}\label{lem:only_positive_integers_less_than_ell}
  Let $S$ be a bi-UFS for which $\ell = \min(\atoms^+(S) \setminus \{1\})$ exists. Then every $s\in S$ satisfying $0 < s < \ell$ is a positive integer.
\end{lemma}

\begin{proof}
  Since $s<\ell$, every divisor of $s$ that is an atom must be $1$. Hence every additive factorization of $s$ only contains $1$, so $s \in \nn$.
\end{proof}

A similar constraint holds just above $\ell$ for non-integer elements.

\begin{lemma}\label{lem:non_integers_less_than_ell_plus_one_are_atoms}
    Let $S$ be a bi-UFS for which $\ell = \min(\atoms^+(S) \setminus \{1\})$ exists. Suppose $r\in S$ is not an integer and $r<\ell+1$. Then $r$ is an additive atom of $S$.
\end{lemma}

\begin{proof}
    Since $r < \ell+1 < 2\ell$, every additive factorization of $r$ contains at most one additive atom different from $1$. Moreover, the additive factorization of $r$ must contain at least one additive atom different from $1$ because $r$ is not a positive integer. Therefore the additive factorization of $r$ contains exactly one additive atom $b \neq 1$, so we can write $r = b + k$ for some nonnegative integer $k<r$.
    If $k\neq0$ then $k\ge 1$, so \[0< b = r - k <(\ell+1)-1=\ell,\] contradicting~\Cref{lem:only_positive_integers_less_than_ell} since $r$ not being an integer implies $b$ is not an integer. Therefore $k = 0$, so $r \in \atoms^+(S)$.
\end{proof}

Combining the previous two lemmas, we show that certain elements $n + \ell$ must be multiplicative atoms.

\begin{lemma}\label{lem:n_plus_ell_is_a_multiplicative_atom}
  Let $S$ be a bi-UFS for which $\ell = \min(\atoms^+(S) \setminus \{1\})$ exists. Let $n$ be a nonnegative integer. If $n + \ell < \ell^2$ then $n + \ell$ is an atom in the multiplicative monoid $S^*$. In particular, $\ell$ is a multiplicative atom.
\end{lemma}

\begin{proof}
  Suppose there exists nonunits $x, y\in S^*$ such that $n + \ell = x y$. If one of $x$ and $y$ is an integer then it would be at least $2$, a contradiction by~\Cref{lem:integers_do_not_divide_atom_plus_integer}. If neither factor is an integer then both must be at least $\ell$ by~\Cref{lem:only_positive_integers_less_than_ell}. Then $x y \ge \ell^2 > n + \ell$, a contradiction. Therefore $n + \ell$ has no nontrivial multiplicative factorization, so $n + \ell$ is an atom. Taking $n = 0$ yields that $\ell$ is a multiplicative atom.
\end{proof}

A similar argument applies to ordinary primes.

\begin{lemma}\label{lem:primes_are_multiplicative_atoms}
    Let $S$ be a bi-UFS for which $\ell = \min(\atoms^+(S) \setminus \{1\})$ exists. Then all positive rational primes $p$ that are less than $\ell$ are multiplicative atoms in $S^\ast$. Additionally, when $\ell>\sqrt2$, the element $2$ is a multiplicative atom in $S^\ast$.
\end{lemma}

\begin{proof}
    Suppose $p$ factors into $xy$ for $x,y \in S^*$. By~\Cref{lem:no_elements_between_zero_and_one}, there are no elements of $S$ in $(0,1)$, so $1 \le x,y \le p<\ell$. Furthermore, by~\Cref{lem:only_positive_integers_less_than_ell}, all elements less than $\ell$ are positive integers, so $x,y \in \nn$. But $p$ is prime in $\nn$, so one of $x,y$ must be $1$. Therefore $p$ has no nontrivial factorization in $S^*$, so it is a multiplicative atom. 
    Suppose $xy=2$ and $\ell>\sqrt2$. At least one of $x$ and $y$ has to be at most $\sqrt 2$, which means that it must be a positive integer less than $\sqrt2$ by~\Cref{lem:only_positive_integers_less_than_ell}. Hence $2$ only has trivial factorizations in $S^*$ in this case.
\end{proof}

We now exclude successive intervals for $\ell$. The algebraic identity \[(1 + x^3)(1 + x + x^2) = (1 + x)(1 + x^2 + x^4),\] already used in the proof of~\Cref{lem:no_elements_between_zero_and_one}, again drives the argument.

\begin{lemma}\label{lem:ell_is_not_between_phi_and_2}
  Let $S$ be a bi-UFS for which $\ell = \min(\atoms^+(S) \setminus \{1\})$ exists. Then $\ell$ cannot satisfy $(1+\sqrt{5})/2=\phi < \ell < 2$.
\end{lemma}

\begin{proof}
  Assume for the sake of contradiction that $\phi < \ell < 2$. Since $\ell > \phi$ and $\phi$ is the positive root of $x^2 = x+1$, the inequality $1 + \ell < \ell^2$ holds. By~\Cref{lem:n_plus_ell_is_a_multiplicative_atom},  $1 + \ell$ is a multiplicative atom of $S^\ast$, and hence prime. We now apply the identity
  \[(1 + \ell^3)(1 + \ell + \ell^2) = (1 + \ell)(1 + \ell^2 + \ell^4).\]
  Because $1 + \ell$ is prime, it divides one of the two factors on the left. This splits the argument into two cases.
  
 \medskip\noindent\textsc{Case A}: Suppose $1 + \ell$ multiplicatively divides $ 1+\ell^3$. Then the quotient $(1 + \ell^3)/(1 + \ell) = \ell^2 - \ell + 1$ is in $S$. Adding $\ell$ to both sides yields $(\ell^2 - \ell + 1) + \ell = \ell^2 + 1$ is in $S$. The left-hand side contains an explicit copy of the additive atom $\ell$, and thus by unique factorization, the additive atomic decomposition of $\ell^2 + 1$ must contain it as well. The summand $1$ on the right is already an atom on its own, so the $\ell$ must come from the decomposition of $\ell^2$. Therefore $\ell^2 - \ell \in S$. We can rearrange $\ell^2 > \ell + 1$ to be $1 < \ell^2 - \ell$, and $\ell < 2$ implies $\ell^2 - \ell < \ell$. Therefore $1 < \ell^2 - \ell < \ell<2$, which contradicts~\Cref{lem:only_positive_integers_less_than_ell} since no integers are in the interval $(1, 2)$. Thus Case A is impossible.

\medskip\noindent\textsc{Case B}: Suppose $1 + \ell$ multiplicatively divides $1 + \ell + \ell^2$. Then the quotient $r = (1 + \ell + \ell^2)/(1 + \ell)$ must lie in $S$. Since $(1+x+x^2)/(1+x)$ is increasing for positive $x$, we know that since $\phi < \ell < 2$, it follows that $$\frac{1+\phi+\phi^2}{1+\phi}=\frac{2\phi^2}{\phi^2} = 2 < r < \frac{7}{3}.$$ Additionally, since $r = \ell + 1/(1+\ell)$ and $1/(1+\ell)<1$, we know $r<\ell+1$. Therefore by~\Cref{lem:non_integers_less_than_ell_plus_one_are_atoms}, we know that $r$ is an additive atom.

    Rearranging and expanding $r = (1 + \ell + \ell^2)/(1 + \ell)$ yields $r + r \ell = 1 + \ell + \ell^2$. Since $r$ is larger than both $1$ and $\ell$, the uniqueness of additive factorizations forces $r$ to appear in the additive factorization of $\ell^2$. Hence we may write $\ell^2 = r + s$ for some $s \in S$. As a real number, $s = \ell^2 - \ell - 1/(1 + \ell)$. 
    Since $\ell < 2$, the inequality $\ell(\ell-2)< 0 < 1/(1+\ell)$ holds, so $s < \ell$. Similarly $\ell > \phi$ implies that $\ell(\ell^2 -1) > \phi(\phi^2-1)=\phi \cdot \phi > 1$ so $s>0$. Since $s \in (0, \ell)$, we know by~\Cref{lem:only_positive_integers_less_than_ell} that $s$ must be a positive integer, and the only integer in $(0, \ell) \subseteq (0,2)$ is $s = 1$.
    
    Substituting this value of $s$ yields the equation $\ell^2 = r + 1 = 1 + \ell + 1/(\ell + 1)$, and after multiplying by $\ell + 1$ on both sides, the equation can be rearranged into $\ell^3 = 2(\ell + 1)$. Since $\ell>\phi>\sqrt2$, we know by~\Cref{lem:primes_are_multiplicative_atoms} that $2$ is a multiplicative atom. Since $\ell$ and $\ell + 1$ are also multiplicative atoms by~\Cref{lem:n_plus_ell_is_a_multiplicative_atom}, we see that this case contradicts that $S^*$ is a UFM.

    None of the cases above are valid, and therefore, $\ell$ cannot satisfy $\phi < \ell < 2$. 
\end{proof}

Rescaling the same identity rules out the next interval.

\begin{lemma}\label{lem:ell_is_not_between_2_and_1+sqrt2}
  Let $S$ be a bi-UFS for which $\ell = \min(\atoms^+(S) \setminus \{1\})$ exists. Then $\ell$ cannot satisfy $2 < \ell < 1 + \sqrt{2}$.
\end{lemma}

\begin{proof}
  Assume for the sake of contradiction that $2 < \ell < 1 + \sqrt{2}$. In particular, because $\ell > 2$, it follows that $\ell^2 > 2 \ell > \ell + 2$. By~\Cref{lem:n_plus_ell_is_a_multiplicative_atom}, $\ell + 2$ is a multiplicative atom of $S^\ast$, and hence prime. We now apply the following identity: 
  \[
    (8 + \ell^3)(4 + 2\ell + \ell^2) = (2 + \ell)(16 + 4 \ell^2 + \ell^4).
  \]
  Since $\ell + 2$ is prime, it divides one of the two factors on the left. We now split this into cases.

\textsc{\medskip\noindent\textsc{Case A}}: Suppose $\ell + 2$ multiplicatively divides $8 + \ell^3$. In this case, the quotient is $$\frac{8 + \ell^3}{2+\ell} = 4 - 2\ell + \ell^2,$$ which therefore lies in $S$. Adding $2\ell$ yields the equation $(4 - 2\ell + \ell^2) + 2\ell = 4 + \ell^2$. The left-hand side exhibits at least two copies of the additive atom $\ell$, so by the uniqueness of additive factorization, the decomposition of the right-hand side must also contain two copies of $\ell$. Both copies must come from $\ell^2$, as $4$ is a positive integer and thus has only the additive atom $1$ in its unique factorization. Thus $\ell^2 - 2 \ell \in S$. Since $2 < \ell < 1 + \sqrt 2$, note that $$0 < \ell^2 - 2\ell = \ell(\ell-2) < (1 + \sqrt 2)(\sqrt 2 - 1) = 1,$$ which contradicts~\Cref{lem:no_elements_between_zero_and_one}.

\textsc{\medskip\noindent\textsc{Case B}}: Suppose $\ell + 2$ multiplicatively divides $4 + 2\ell + \ell^2$. The quotient $$r:=\frac{4 + 2\ell + \ell^2}{2 + \ell}=\ell + \frac{4}{2 + \ell}$$ must lie in $S$. Since $x+4/(2+x)$ is increasing for positive $x$, we know that since $2 < \ell < 1 + \sqrt 2$, it follows that $$3 < r < 1+\sqrt2+\frac{4}{3+\sqrt2}=\frac{19}{7}+\frac{3\sqrt2}{7}<4,$$ so $r$ is not an integer. Additionally, since $4/(2+\ell)<1$ for $\ell > 2$, we know $r<\ell+1$. Therefore by~\Cref{lem:non_integers_less_than_ell_plus_one_are_atoms}, we know $r \in \atoms^+(S)$. Rearranging and expanding $r = (4 + 2 \ell + \ell^2)/(2 + \ell)$ yields $2r + \ell r= 4 + 2 \ell + \ell^2$. The additive atom $r$ appears at least twice in the atomic decomposition of the left-hand side. By uniqueness of additive factorizations, the same two copies must also appear on the right. Since $r$ is larger than both $1$ and $\ell$, and the additive atomic decomposition of $4+2\ell$ contains only $1$'s and $\ell$'s, the additive atomic decomposition of $\ell^2$ must contain both copies of $r$. Hence $\ell^2 \ge 2 r$. However, $\ell^2 < (1+\sqrt 2)^2 < 6 < 2r$, a contradiction. 

The above cases both yield contradictions, so therefore $\ell$ may not satisfy $2 < \ell < 1 + \sqrt 2$. 
\end{proof}

For the final interval we use a parametric identity, choosing an integer $n \in (\ell, \ell^2 - \ell)$.

\begin{lemma}\label{lem:ell_is_not_at_least_1+sqrt2}
  Let $S$ be a bi-UFS for which $\ell = \min(\atoms^+(S) \setminus \{1\})$ exists. Then $\ell$ cannot satisfy $\ell \ge 1 + \sqrt{2}$.
\end{lemma}

\begin{proof}
  Assume for the sake of contradiction that $\ell \ge 1 + \sqrt{2}$. Consider the open interval $(\ell, \ell^2 - \ell)$. When $\ell = 1 + \sqrt{2}$, this interval is $(1 + \sqrt 2, 2 + \sqrt 2)$, which contains the integer $3$. For $\ell > 1 + \sqrt 2$, we know that the length of the interval $(\ell, \ell^2 - \ell)$, which is $\ell^2-2\ell$, will be larger than $(1+\sqrt2)^2-2(1+\sqrt2)=1$ because $x^2-2x$ is increasing for all $x\geq1+\sqrt{2}$. Since all open intervals that are larger than $1$ in length contain an integer, we know that the interval $(\ell, \ell^2 - \ell)$ will always contain an integer. Fix that integer as $n$. Then $n + \ell < \ell^2$, so by~\Cref{lem:n_plus_ell_is_a_multiplicative_atom}, we know $n + \ell$ is an atom in the multiplicative monoid and hence prime. We now apply the following identity, similar to the above theorems:
  \[
    (n^3 + \ell^3)(n^2 + n \ell + \ell^2) = (n + \ell)(n^4 + n^2 \ell^2 + \ell^4).
  \]
  We now divide into cases based on which factor the prime $n + \ell$ divides on the left-hand side.
  
  \textsc{\medskip\noindent\textsc{Case A}}: Suppose $n + \ell$ multiplicatively divides $n^3 + \ell^3$. Then the quotient $q := n^2 - n \ell + \ell^2$ is an element of $S$. Adding $n \ell$ to both sides of the equation yields $q + n \ell = n^2 + \ell^2$. The left-hand side contains $n$ copies of the additive atom $\ell$, which must also appear in the additive atomic decomposition of the right-hand side by unique factorization. Since $n^2$ is an integer and hence only has $1$ in its additive atomic decomposition, we know that the additive atomic decomposition of $\ell^2$ must contain $n$ copies of $\ell$. Thus $\ell^2 \ge n\ell$. However, $n > \ell$ by construction, so $n \ell > \ell^2$, a contradiction. 

  \textsc{\medskip\noindent\textsc{Case B}}: Suppose $n + \ell$ multiplicatively divides $n^2 + n \ell + \ell^2$. Then $$r := \frac{n^2 + n \ell + \ell^2}{n + \ell} = n + \frac{\ell^2}{n + \ell}$$ lies in $S$. Let the unique additive atomic factorization of $\ell^2$ be of the form
  \[
    \ell^2 = m_0 \cdot 1 + m_\ell \ell + \sum_{b \in \mathcal{B}} m_b b,
  \]
  where $\mathcal{B}$ consists of all additive atoms other than $1$ and $\ell$ in the factorization of $\ell^2$. Since every additive atom not equal to $1$ is at least $\ell$, each coefficient $m_b$ and $m_\ell$ is at most $\ell$, and hence less than $n$. Write the unique additive atomic factorization of $r$ as
  \[
    r = c_0 \cdot 1 + c_\ell \ell + \sum_{b \in \mathcal{C}} c_b b,
  \]
  where $\mathcal{C}$ is the set of additive atoms other than $1$ and $\ell$ in the factorization of $r$. By the definition of $r$,
    \begin{equation} \label{eq:equation_with_nr}
		nr + \ell r = n^2 + n \ell + \ell^2.
	\end{equation}
  For each atom $b\notin\{1,\ell\}$, the term $n r$ contributes $n c_b$ copies of $b$ to the left-hand side of~\eqref{eq:equation_with_nr}. Since the unique additive factorization of $n^2 + n \ell$ is $n^2$ copies of the atom $1$ and $n$ copies of the atom $\ell$, we know that the $n c_b$ copies of the atom $b$ must be in the additive factorization of $\ell^2$. Since $b > \ell$, we know $\ell^2 < b\ell < bn$. Hence the right-hand side of~\eqref{eq:equation_with_nr} contains fewer than $n$ copies of $b$. This implies that $c_b = 0$ for all $b \not \in \{ 1, \ell \}$, and that $r = c_0 + c_\ell \ell$. Hence $nr = nc_0 + nc_\ell \ell$ and $\ell r = c_0\ell + c_\ell \ell^2$. Therefore
  \begin{equation} \label{eq:equation_expanded}
		nc_0 + nc_\ell \ell + c_0\ell + c_\ell \ell^2 = n^2 + n \ell + \ell^2.
  \end{equation}

  We now count multiplicities of $\ell$ in~\eqref{eq:equation_expanded}. On the left-hand side, there are $nc_\ell + c_0 + c_\ell m_\ell$ copies of $\ell$, while the right-hand side contains $n + m_\ell$ copies of $\ell$. For $c_\ell \ge 1$, the inequality $nc_\ell + c_0 + c_\ell m_\ell > n + m_\ell$ holds, so necessarily $c_\ell = 0$ and $r = c_0$, an integer. Thus, we know from $r = n + \ell^2/(n + \ell)$ that $h := \ell^2/(n + \ell)$ is a positive integer and hence in $S$. Because $n + \ell < \ell^2$ and $n > \ell$, we know $1 < h < \ell / 2 < \ell$. Thus $h \ge 2$ is an integer strictly less than $\ell$. We can factor $h$ into rational primes, each of which are less than $\ell$ and hence are multiplicative atoms by~\Cref{lem:primes_are_multiplicative_atoms}. Therefore the equation $\ell \cdot \ell = h(n + \ell)$ describes two distinct factorizations, since $\ell$ and $n + \ell$ are multiplicative atoms by~\Cref{lem:n_plus_ell_is_a_multiplicative_atom}. This is a contradiction, so the desired result follows.
\end{proof}

It follows from the above lemmas that there is only one possible interval left for $\ell$.

\begin{lemma} \label{lem:final_containment}
Let $S$ be a bi-UFS. Suppose $\ell:=\min(\atoms^+(S)\setminus\{1\})$ exists. Then we must have $1 < \ell \le \phi=(1+\sqrt{5})/2$, the golden ratio.
\end{lemma}

\begin{proof}
From~\Cref{lem:ell_is_not_between_phi_and_2},~\Cref{lem:ell_is_not_between_2_and_1+sqrt2},~\Cref{lem:ell_is_not_at_least_1+sqrt2}, and the fact that $2$ is not an additive atom implies $\ell\neq2$, we see that the only remaining interval for $\ell$ is $1 < \ell \le \phi$.
\end{proof}

\subsection{Proof of the Degree-Two Case} To prove the main theorem of this section, we use the fact that the following two positive semidomains are not bi-UFS.

\begin{example}\label{ex:golden_ratio}
    Let $\phi=(1+\sqrt{5})/2$ be the golden ratio, which satisfies $1+\phi=\phi^2$. We will show the positive semidomain $\nn_0[\phi]$ is not bi-UFS. We first prove that $5$ is a multiplicative atom. If $$(a+b\phi)(c+d\phi)=ac+\phi(ad+bc)+(1+\phi)(bd)=5$$ then $ad+bc+bd=0$, so $ad=bc=bd=0$. Suppose for the sake of contradiction that at least one of $b$ and $d$ is nonzero; without loss of generality assume $b\neq0$. Then $d=0$ and $c=0$, a contradiction. Hence $b=d=0$, so $ac=5$, but $5$ has no nontrivial factorization into integers. Hence $5(1+\phi)=(2+\phi)(2+\phi)$ implies $\nn_0[\phi]^{\ast}$ is not UF, as the atomic decomposition of $5(1+\phi)$ in $\nn_0[\phi]^*$ will include $5$, while $5>2+\phi$ implies the decomposition of $(2+\phi)(2+\phi)$ will not.
\end{example}

\begin{example}\label{ex:sqrt2}
    We will show that the positive semidomain $\nn_0[\sqrt2]$ is not bi-UFS. We first prove that $3+\sqrt2\in\atoms^\times(\nn_0[\sqrt2])$. Suppose $(a+b\sqrt{2})(c+d\sqrt{2})=3+\sqrt2$ for nonnegative integers $a$, $b$, $c$, and $d$. Then we must have $ac+2bd=3$ and $ad+bc=1$. We must have $a,c\neq0$ because $2bd=3$ has no solution. Therefore at least one of $b$ and $d$ is equal to $0$ because exactly one of $ad$ and $bc$ is $0$, so $bd=0$. Then the equation $ac=3$ implies either $a$ or $c$ equals $1$. Suppose without loss of generality that $a=1$, so $c=3$. If $b>0$ then $ad+bc>1$, contradiction. Thus we must have $b=0$. However, this implies $a+b\sqrt{2}=1$, showing that every factorization of $3+\sqrt2$ must be trivial. Consider the equation $7(1+\sqrt2)=(1+2\sqrt2)(3+\sqrt2)$. Since $(3+\sqrt{2})(3-\sqrt{2})=7$ and $3-\sqrt{2}\notin\nn_0[\sqrt2]$, we know $3+\sqrt2$ does not multiplicatively divide $7$. We also know $3+\sqrt2$ does not multiplicatively divide $1+\sqrt{2}$ since $3+\sqrt2>1+\sqrt{2}$ and $\nn_0[\sqrt2]\cap(0,1)=\emptyset$. Since the decomposition of $7(1+\sqrt2)$ into atoms will not include the atom $(3+\sqrt2)$, but that of $(1+2\sqrt2)(3+\sqrt2)$ does, we know that $\nn_0[\sqrt{2}]^\ast$ is not a UFM.
\end{example}

Finally, we prove the main theorem of this section.

\begin{theorem} \label{thm:deg2}
Let $\alpha>0$ be algebraic of degree 2. Then $\nn_0[\alpha]$ is not bi-UFS.
\end{theorem}

\begin{proof}
Suppose for the sake of contradiction $\nn_0[\alpha]$ is bi-UFS for some positive algebraic number $\alpha$ of degree $2$. By~\Cref{thm:equivalent_conditions_UFM}, since $\nn_0[\alpha]$ is a UFM under addition, $m_\alpha(x)=x^2-ax-b$ for nonnegative integers $a$ and $b$. By~\Cref{lem:no_elements_between_zero_and_one}, we know that $\alpha>1$. Hence the least additive atom of $\nn_0[\alpha]$ larger than 1 is $\alpha$, so it must be true that $1<\alpha\le \phi$ from~\Cref{lem:final_containment}. If $a=0$, the inequality $1<\alpha\le \phi$ forces $b=2$, but Example~\ref{ex:sqrt2} shows that $\nn_0[\sqrt2]$ is not bi-UFS. If $b=0$ then $m_\alpha(x)$ is reducible and hence not a minimal polynomial. If $a,b\geq1$ then because $\alpha^2-a\alpha-b=0$, it follows that $\alpha^2-\alpha-1\geq0$, so $\alpha\ge\phi$, which only allows $\alpha=\phi$. However, Example~\ref{ex:golden_ratio} shows that $\nn_0[\phi]$ is not bi-UFS. Therefore $\nn_0[\alpha]$ cannot be bi-UFS when $\alpha$ is a positive algebraic number of degree 2.
\end{proof}

\bigskip

\section{Algebraic Monogenic Semidomains of Degree at Least Three}\label{sec:deg3up}

In this section we extend the degree-two result to all algebraic
degrees at least $3$. We proceed in two stages. First, we
establish two lemmas about $\nn_0[\alpha]$ that require only
the additive monoid to be a UFM. These let us extract structural
information about $q_\alpha(x)$ from the failure of certain elements
to be multiplicative atoms. Second, we use these tools to restrict
$m_\alpha(x)$ down to a short list of candidate forms, and then
eliminate each candidate in turn.

\subsection{Preliminary Lemmas} We begin with two lemmas that place no assumption on the multiplicative monoid $\nn_0[\alpha]^\ast$, not even atomicity. The first identifies a sufficient condition for an element
of $\nn_0[\alpha]$ of small additive length to fail to be a
multiplicative atom.

\begin{lemma}\label{lem:alpha_divides_F}
    Let $\alpha$ be a positive algebraic number such that $(\nn_0[\alpha], +)$ is a UFM. 
    Suppose $F\in \nn_0[\alpha]$ has an additive length of at most $3$, is not an atom of $\nn_0[\alpha]^*$, and is not a unit of $\nn_0[\alpha]^*$. Then $\alpha$ multiplicatively divides $F$.
\end{lemma}

\begin{proof}
    Since $F$ is not a multiplicative atom nor a multiplicative unit, we can write $F=XY$ for $X,Y\in\nn_0[\alpha]$ that are not multiplicative units. Since $|F|\leq 3$ and $|F|=|XY|\geq |X||Y|$ by~\Cref{lem:length_inequality}, one of $X,Y$ has additive length $1$. Without loss of generality, suppose $X$ has additive length $1$, so $X=\alpha^r$ for some $0\leq r\leq d-1$ because $\atoms^+(\nn_0[\alpha])=\{1,\alpha,\ldots,\alpha^{d-1}\}$ by~\Cref{thm:equivalent_conditions_UFM}. If $r=0$ then $X=1$, a multiplicative unit. Thus, we know $r\geq 1$. Hence $\alpha$ multiplicatively divides $F$.
\end{proof}

Applying~\Cref{lem:alpha_divides_F}, the next lemma uses the information about a single element $F$ to produce a coefficient bound on $q_\alpha(x)$.

\begin{lemma}\label{lem:conditions_on_a_non_atom_to_classification_of_q_alpha}
    Let the additive monoid of $\nn_0[\alpha]$ be a UFM for some positive algebraic $\alpha$. Suppose $F\in \nn_0[\alpha]$ is additively divisible by $1$, has an additive length of at most $3$, is not an atom of $\nn_0[\alpha]^*$, and is not a unit of $\nn_0[\alpha]^*$. Then the following statements hold.
    \begin{enumerate}
        \item The sum of the coefficients in $q_\alpha(x)$ is less than $4$.
        \smallskip
        \item If every coefficient in $F$ is less than $s$ for some positive integer $s\geq2$ then every coefficient in $q_\alpha(x)$ is less than $s$.
    \end{enumerate}
\end{lemma}

\begin{proof}
    Let $d$ be the degree of the minimal polynomial of $\alpha$. We know by~\Cref{thm:equivalent_conditions_UFM} that since $\nn_0[\alpha]$ is a UFM under addition, we know $\atoms^+(\nn_0[\alpha])=\{1,\alpha,\ldots,\alpha^{d-1}\}$, the minimal polynomial $m_\alpha(x)=x^d-q_\alpha(x)$, where
    \[
        q_\alpha(x)=a_{d-1}x^{d-1}+\cdots+a_1x+a_0\in\nn_0[x],
    \]
    and the additive length of an element $a\in \nn_0[\alpha]$ is the sum of the coefficients of the unique polynomial in $\nn_0[x]$ with degree at most $d-1$ that evaluates to $a$ when $x=\alpha$. By~\Cref{lem:alpha_divides_F}, we know that there exists some $Y\in\nn_0[\alpha]$ and integer $r$ satisfying $1\leq r\leq d-1$ such that $F=\alpha^rY$. We may write $Y=\sum_{i=0}^{d-1}c_i\alpha^i$ for $c_i\in\nn_0$. If $r+i<d$ for every $i$ with $c_i>0$ then $F=\sum_{i=0}^{d-1}c_i\alpha^{r+i}$ is already written with exponents below $d$ and every exponent is positive, so the constant coefficient of $F$ is $0$. However $1$ additively divides $F$, and because all atoms are primes in a UFM, then $1$ additively divides one of the atoms from $\{\alpha,\alpha^2,\ldots,\alpha^{d-1}\}$, a contradiction. Therefore there is some $i$ with $c_i>0$ such that $r+i\geq d$. Fix such an $i$ and let $n=r+i$.
    
\smallskip
    \noindent\emph{(1).} Suppose for the sake of contradiction that $a_0+a_1+\cdots+a_{d-1}\geq 4$. We may reduce the exponents in the term $c_i\alpha^n$ as $c_i\alpha^n=c_i\alpha^{n-d}q(\alpha)$.
    The sum of the coefficients on the right hand side is
        \[
            c_i(a_0+a_1+\cdots+a_{d-1})\geq4.
        \]
        Now consider a later reduction step. A term $M\alpha^e$ with $e\geq d$ is replaced by
        \[ M\alpha^{e-d}q(\alpha)=\sum_{h=0}^{d-1}Ma_h\alpha^{e-d+h}.\]
        This changes the sum of coefficients contributed by that term from $M$ to
        \[ M(a_0+a_1+\cdots+a_{d-1}),\]
        which is at least $M$. Therefore after enough reductions to bring all exponents on $\alpha$ to be at most $d-1$, the term $c_i\alpha^n$ of $Y$ contributes an additive length of at least $4$. Therefore $|F| \ge 4$, a contradiction. Hence we must have $a_0+a_1+\cdots+a_{d-1}<4$.
        
        \smallskip
        \noindent\emph{(2).} Suppose for the sake of contradiction that there exists a positive integer $s\geq2$ such that $q_\alpha(x)$ has a coefficient that is at least $s$ and every coefficient in $F$ is less than $s$. We may choose some $j$ with $0\leq j\leq d-1$ and $a_j\geq s$. We will show that for every $n\geq d$, the expression of $\alpha^n$ as a sum of $1,\alpha,\ldots,\alpha^{d-1}$ has some coefficient at least $s$. Indeed, reducing once gives
        \[
            \alpha^n=\alpha^{n-d}q_\alpha(\alpha)=\sum_{h=0}^{d-1}a_h\alpha^{n-d+h}.
        \]
        In particular, this contains the term $a_j\alpha^{n-d+j}$. If $n-d+j<d$ then we already have a coefficient at least $a_j\geq s$ on a term of degree less than $d$. Otherwise, we may reduce this term again. Among the terms produced is $a_j^2\alpha^{n-2d+2j}$. Continuing like this after $m$ such steps, a term $a_j^m\alpha^{n-md+mj}$ occurs. The exponent drops by $d-j\geq1$ at each step, so for some $m\ge 1$, the exponent $n-md+mj$ is between $0$ and $d-1$, with a coefficient $a_j^m\geq s$. Since every coefficient in every reduction is nonnegative, no cancellation can remove this coefficient.
    
        Applying this to $c_i\alpha^n$, we know writing $c_i\alpha^n$ as a sum of $1,\alpha,\ldots,\alpha^{d-1}$ will have some coefficient at least $s$ because $n \geq d$. Adding the other terms of $XY$ only adds nonnegative coefficients so writing $XY$ as a sum of $1,\alpha,\ldots,\alpha^{d-1}$ will have some coefficient at least $s$, a contradiction. Hence no such $s$ may exist.
\end{proof}

\subsection[Proof of the degree at least 3 case]{Proof of the degree-\texorpdfstring{$\ge 3$}{at least 3} case}
We now turn to the bi-UFS hypothesis and use it together with the
above lemmas to constrain $m_\alpha(x)$. The main tool is a non-atom in $\nn_0[\alpha]^\ast$ supplied by an algebraic identity, generalizing the identity used throughout the $\ell$-analysis of the previous section.

\begin{lemma}\label{lem:1+alpha^r_not_multiplicative_atom}
    Suppose $\nn_0[\alpha]$ is bi-UFS for some positive algebraic $\alpha$ of degree $d$. For every positive integer $r$ satisfying $2r<d$, at least one of $1+\alpha^r+\alpha^{2r}$ and $1+\alpha^r$ is not a multiplicative atom.
\end{lemma}

\begin{proof}
    Consider the algebraic identity
    \[
        (1+\alpha^r+\alpha^{2r})(1+\alpha^{3r})=(1+\alpha^r)(1+\alpha^{2r}+\alpha^{4r}).
    \]
    Suppose for the sake of contradiction that both $1+\alpha^r+\alpha^{2r}$ and $1+\alpha^r$ are multiplicative atoms. Since $1+\alpha^r\neq1+\alpha^r+\alpha^{2r}$ and $\nn_0[\alpha]^*$ is reduced by~\Cref{lem:multiplicative_monoid_reduced}, we know $1+\alpha^r$ does not multiplicatively divide $1+\alpha^r+\alpha^{2r}$. Since $\nn_0[\alpha]^\ast$ is a UFM, we know $1+\alpha^r$ is prime in $\nn_0[\alpha]^\ast$. Therefore $1+\alpha^r$ must multiplicatively divide $1+\alpha^{3r}$, so $\alpha^{2r}-\alpha^r+1\in \nn_0[\alpha]$. Since $(\alpha^{2r}-\alpha^r+1)+\alpha^r=\alpha^{2r}+1$, we know that $\alpha^r$ additively divides $\alpha^{2r}+1$.
    
    Since the additive monoid $\nn_0[\alpha]$ is a UFM and $2r<d$, we know by~\Cref{thm:equivalent_conditions_UFM} that $\{1,\alpha^r,\alpha^{2r}\}\subseteq \atoms(\nn_0[\alpha])$. However, since $\alpha^r$ additively divides $\alpha^{2r}+1$, this contradicts that the additive monoid is a UFM. Therefore at least one of $1+\alpha^r+\alpha^{2r}$ and $1+\alpha^r$ is not a multiplicative atom.
\end{proof}

By~\Cref{lem:1+alpha^r_not_multiplicative_atom} and the coefficient bound from~\Cref{lem:conditions_on_a_non_atom_to_classification_of_q_alpha}, we restrict $q_\alpha(x)$ to a short list of polynomials.

\begin{lemma}\label{lem:deg_3_plus_min}
    Let positive $\alpha$ be algebraic of degree $d\ge 3$, and suppose $\nn_0[\alpha]$ is bi-UFS. Then
    \[ m_\alpha(x)=x^d-x^i-1 \hspace{.5 cm} \text{or} \hspace{.5 cm}  m_\alpha(x)=x^d-x^j-x^i-1\]
    for some $1\leq i<j\leq d-1$.
\end{lemma}

\begin{proof}
Suppose $S=\nn_0[\alpha]$ is bi-UFS. We know by~\Cref{thm:equivalent_conditions_UFM} that since $\nn_0[\alpha]$ is a UFM under addition, we know $\atoms^+(\nn_0[\alpha])=\{1,\alpha,\ldots,\alpha^{d-1}\}$ and the minimal polynomial $m_\alpha(x)=x^d-q_\alpha(x)$, where
    \[
        q_\alpha(x)=a_{d-1}x^{d-1}+\cdots+a_1x+a_0\in\nn_0[x].
    \]

By~\Cref{lem:conditions_on_a_non_atom_to_classification_of_q_alpha}, we know that since $1+\alpha+\alpha^2$ and $1+\alpha$ are both additively divisible by $1$, both have an additive length of at most $3$ since $d\ge 3$, both have every coefficient less than $2$, and both are not multiplicative units because of~\Cref{lem:multiplicative_monoid_reduced}, and at least one of them is not a multiplicative atom by~\Cref{lem:1+alpha^r_not_multiplicative_atom} since $2<d$, then the sum of the coefficients of $q_\alpha(x)$ is less than $4$ and every coefficient in $q_\alpha(x)$ is less than $2$. Since $m_\alpha$ is irreducible, we know $a_0\geq1$, so $a_0=1$. If the sum of the coefficients of $q_\alpha(x)$ is $1$ then $q_\alpha(x)=1$. Then $\alpha^d=1$, forcing $\alpha=1$, a contradiction. Hence the sum of the coefficients of $q_\alpha(x)$ is $2$ or $3$. Therefore $q_\alpha(x)=1+x^i$ or $q_\alpha(x)=1+x^i+x^j$ for some $1\leq i<j\leq d-1$.
\end{proof}

The next two lemmas show when $\alpha$ fails to divide an element of small additive length, which we then use to derive a contradiction.

\begin{lemma}\label{lem:condition_for_alpha_not_multiplicative_divisor}
    Let $\nn_0[\alpha]$ be bi-UFS for some positive algebraic $\alpha$ of degree $d$. Suppose $q_\alpha(x)$ only contains coefficients less than $2$. Let $F\in\nn_0[\alpha]$ be written as $F=f_0+f_1\alpha+\cdots+f_{d-1}\alpha^{d-1}$ for nonnegative integer coefficients $f_i$. If some $s\in \supp(q_\alpha(x)) \setminus \{0\}$ satisfies $f_s<f_0$ then $\alpha$ is not a multiplicative divisor of $F$.
\end{lemma}

\begin{proof}
    Suppose for the sake of contradiction that $F$ is multiplicatively divisible by $\alpha$. Then
    \[
        F=\alpha(c_0+c_1\alpha+\cdots+c_{d-1}\alpha^{d-1})
    \]
    for nonnegative integer coefficients $c_i$. Let $S:=\supp(q_\alpha(x))\setminus \{0\}$. Since $\alpha^d=1+\sum_{s\in S}\alpha^s$, we know
    \[
        F=c_{d-1}\left(1+\sum_{s\in S}\alpha^s\right)+c_0\alpha+c_1\alpha^2+\cdots+c_{d-2}\alpha^{d-1}.
    \]
    Hence $c_{d-1}=f_0$ since the coefficients $f_i$ are unique because the additive monoid $\nn_0[\alpha]$ is a UFM. Therefore $f_s=c_{s-1}+c_{d-1}\geq f_0$, a contradiction. Consequently, $\alpha$ does not multiplicatively divide $F$.
\end{proof}

\begin{lemma}\label{lem:S_is_subset_of_r_and_2r}
    Let $\nn_0[\alpha]$ be bi-UFS for some positive algebraic $\alpha$ of degree $d$. Suppose $q_\alpha(x)$ only contains coefficients less than $2$. Then for every positive integer $r$ satisfying $2r<d$, the set $\supp(q_\alpha(x)) \setminus \{0 \}$ is a subset of $\{r,2r\}$.
\end{lemma}

\begin{proof}
    Suppose for the sake of contradiction that there exists a positive integer $r$ satisfying $2r<d$ for which $\supp(q_\alpha(x)) \setminus \{0\}$ is not a subset of $\{r,2r\}$. Then there is some $s\in \supp(q_\alpha(x)) \setminus \{0\}$ with $s\notin \{r,2r\}$. Let $F$ be one of the elements of $\{1+\alpha^r,1+\alpha^r+\alpha^{2r}\}$. Since the additive monoid $\nn_0[\alpha]$ is a UFM, we may write $F$ uniquely as \[ F=f_0+f_1\alpha+\cdots+f_{d-1}\alpha^{d-1}\] for nonnegative integer coefficients $f_i$. Note that $f_s=0$ and $f_0=1$ since $2r<d$ implies there is no algebraic manipulation needed to turn $F$ into that form. Since $f_s<f_0$, we know by~\Cref{lem:condition_for_alpha_not_multiplicative_divisor} that $\alpha$ does not multiplicatively divide $F$. The inequality $2r<d$ implies $F$ has an additive length of at most $3$,~\Cref{lem:multiplicative_monoid_reduced} implies $F$ is not a multiplicative unit, and $\alpha$ does not multiplicatively divide $F$, so therefore by~\Cref{lem:alpha_divides_F}, we know $F$ must be a multiplicative atom. However, by~\Cref{lem:1+alpha^r_not_multiplicative_atom}, we know that since $2r<d$, at least one of the elements in $\{1+\alpha^r,1+\alpha^r+\alpha^{2r}\}$ is not a multiplicative atom, a contradiction. Therefore no such $r$ can exist.
\end{proof}

We can now eliminate each candidate form for $m_\alpha(x)$ from~\Cref{lem:deg_3_plus_min}, completing the degree-$\ge 3$ case.

\begin{theorem} \label{thm:deg_3_and_up}
Let $\alpha>0$ be algebraic of degree $d\geq 3$. Then $\nn_0[\alpha]$ is not bi-UFS.
\end{theorem}

\begin{proof}
Assume for the sake of contradiction that $S=\nn_0[\alpha]$ is bi-UFS. By~\Cref{lem:deg_3_plus_min}, the minimal polynomial of $\alpha$ is in the form
\[ m_\alpha(x)=x^d-x^i-1 \hspace{.5 cm} \text{or} \hspace{.5 cm}m_\alpha(x)=x^d-x^j-x^i-1,\]
for integers $1\leq i<j\leq d-1$. Also, by~\Cref{lem:n_plus_ell_is_a_multiplicative_atom}, $\alpha$ is an atom, so it must be a prime because $\nn_0[\alpha]^*$ is a UFM. Suppose for the sake of contradiction that $m_\alpha(x)=x^d-x^j-x^i-1$ for some $1\leq i<j\leq d-1$. By~\Cref{lem:S_is_subset_of_r_and_2r}, since $2<d$, we know that $\{i,j\}$ is a subset of $\{1,2\}$ by taking $r = 1$. Therefore $\alpha^d=\alpha^2+\alpha+1$, so $\alpha(1+\alpha^{d-1})=(1+\alpha)^2$. Since $2\in\supp(q_\alpha(x))$, and the coefficient of $\alpha^2$ in $1+\alpha$ is less than the constant term of $1+\alpha$, we know by~\Cref{lem:condition_for_alpha_not_multiplicative_divisor} that $\alpha$ is not a multiplicative divisor of $1+\alpha$. 
This is a contradiction as $\alpha$ is prime. Therefore we must have $m_\alpha(x)=x^d-x^i-1$ for some $1\leq i\leq d-1$. By~\Cref{lem:S_is_subset_of_r_and_2r}, since $2<d$, we know that $i$ must be an element of $\{1,2\}$. Suppose for the sake of contradiction that $i=2$. Then $\alpha^d=\alpha^2+1$, so $\alpha(2+\alpha^{d-1})=(1+\alpha)^2$. By the same reasoning as above, since $2\in\supp(q_\alpha(x))$ and the coefficient of $\alpha^2$ in $1+\alpha$ is less than the constant term of $1+\alpha$, we know by~\Cref{lem:condition_for_alpha_not_multiplicative_divisor} that $\alpha$ is not a multiplicative divisor of $1+\alpha$. In the same way, we arrive at a contradiction since $\alpha$ cannot be prime. Therefore $i=1$. If $d \geq 5$, taking $r = 2$ in~\Cref{lem:S_is_subset_of_r_and_2r} implies $i \in \{2, 4\}$, contradiction. Thus $d \in \{3, 4 \}$.
Suppose towards a contradiction that $d=3$. Then since $\alpha^3=\alpha+1$, we know $(1+\alpha^2)^2=\alpha(\alpha^2+3\alpha)$. Since $1\in\supp(q_\alpha(x))$ and the coefficient of $\alpha$ in $1+\alpha^2$ is less than the constant term of $1+\alpha^2$, we know by~\Cref{lem:condition_for_alpha_not_multiplicative_divisor} that $\alpha$ is not a multiplicative divisor of $1+\alpha^2$. This is a contradiction as $\alpha$ is a prime. 
Therefore $d=4$. Then since $\alpha^4=\alpha+1$, we know $$(1+\alpha^2)(1+\alpha^3)=1+\alpha^2+\alpha^3+(\alpha^2+\alpha)=\alpha(2\alpha+\alpha^2+\alpha^3).$$ Since $1\in\supp(q_\alpha(x))$, the coefficient of $\alpha$ in $1+\alpha^2$ is less than the constant term of $1+\alpha^2$, and the coefficient of $\alpha$ in $1+\alpha^3$ is less than the constant term of $1+\alpha^3$, we know by~\Cref{lem:condition_for_alpha_not_multiplicative_divisor} that $\alpha$ is not a multiplicative divisor of $1+\alpha^2$ nor $1+\alpha^3$. This contradicts $\alpha$ being a prime. All cases lead to a contradiction, so $S=\nn_0[\alpha]$ must not be bi-UFS.
\end{proof}

The degree-two and degree-$\ge 3$ cases together yield the characterization of bi-UFS cyclic algebraic semidomains.

\begin{theorem}\label{thm:final_theorem}
Let $\alpha>0$ be algebraic. The semidomain $\nn_0[\alpha]$ is bi-UFS if and only if $\alpha\in\nn$.
\end{theorem}

\begin{proof}
    Suppose $\nn_0[\alpha]$ is bi-UFS. Then from~\Cref{thm:deg2} and~\Cref{thm:deg_3_and_up}, we know $\alpha$ has to be of degree 1. Since the additive monoid $\nn_0[\alpha]$ is a UFM, we know from~\Cref{thm:equivalent_conditions_UFM} that $m_\alpha(x)=x-n$ for some positive integer $n$. In particular, $\alpha\in\nn$,
    and $\nn_0[\alpha]=\nn_0$, the prototypical bi-UFS.
\end{proof}
\bigskip

\section{Finitely Generated Algebraic Positive Semidomains}\label{sec:multigenerator}

Throughout this section, we consider the semidomain $S = \nn_0[\alpha_1, \ldots, \alpha_n]$, where $\alpha_1, \ldots, \alpha_n$ are positive algebraic numbers. Our goal is to show that $S$ is bi-UFS if and only if $S = \nn_0$, which reduces to showing that if $S$ is a bi-UFS then $S$ is the semidomain of nonnegative integers. The strategy is to embed $S$ inside the rational cone generated by its additive atoms, and then to use a Perron--Frobenius argument to force the cone to be one-dimensional. 

We begin by establishing that if $S$ is a bi-UFS then it contains finitely many additive atoms, which makes the rational cone generated by its additive atoms finite-dimensional.

\begin{lemma}\label{lemma:finitely_atoms}
If the additive monoid of the semidomain $S\coloneqq\nn_0[\alpha_1,\ldots,\alpha_n]$ is a UFM then the set of additive atoms $\atoms^+(S)$ is finite. In particular, the inequality $|\atoms^+(S)|\le [\qq(\alpha_1,\ldots,\alpha_n):\qq]$ holds.
\end{lemma}

\begin{proof}
Assume that $(S,+)$ is a UFM. Let $K = \Q(\alpha_1, \ldots, \alpha_n)$ which contains $S$. Since $\alpha_1, \ldots, \alpha_n$ are algebraic then the extension $K/\Q$ is finite. Suppose towards a contradiction that the additive atoms of $S$ are not linearly independent over $\Q$. Then there exist additive atoms $a_1, \ldots, a_m \in \atoms^+(S)$ and rational numbers $q_1, \ldots, q_m \in \qq$ (not all equal to $0$) such that $q_1 a_1 + \cdots + q_m a_m = 0$. Clearing denominators yields integers $n_1, \ldots, n_m$ (not all zero) such that the equality $n_1a_1 + \ldots + n_ma_m = 0$ holds. Separating positive and negative coefficients yields $\sum_{q_i' > 0} q_i' a_i = \sum_{q_i ' < 0}(-q_i')a_i$. Both sides are additive factorizations of the same element, and since the positive and negative sets are disjoint, these are different factorizations. This contradicts that $(S,+)$ is an additive UFM. Therefore $\atoms^+(S)$ is linearly independent over $\Q$. Since $\atoms^+(S) \subset K$ and $K$ is finite-dimensional over $\Q$, the inequality $|\atoms^+(S)| \leq [K : \Q]$ holds.
\end{proof}

The following lemma records a useful consequence of nonnegativity: if a coordinate subspace is invariant under a sum of linear maps whose matrices have nonnegative entries then it is invariant under each summand. We will use this observation later to prove that certain multiplication matrices are irreducible.

\begin{lemma}\label{lem:each_commp_stabilizes}
Let \(V\) be a vector space over an ordered field with basis
\(e_1,\dots,e_r\). For each subset \(I \subseteq \llbracket 1,r\rrbracket\), set
\[
W_I \coloneqq \operatorname{span}\{e_i \colon i \in I\}.
\]
Let \(T_1,\dots,T_m \colon V \to V\) be linear maps whose matrices with
respect to the basis \(e_1,\dots,e_r\) have nonnegative entries, and set
\[
T \coloneqq \sum_{j=1}^m T_j.
\]
If \(T(W_I) \subseteq W_I\) then \(T_j(W_I) \subseteq W_I\) for every
\(j \in \llbracket 1,m\rrbracket\).
\end{lemma}

\begin{proof}
Suppose that \(T(W_I) \subseteq W_I\), and fix \(i \in I\). For each
\(j \in \llbracket 1,m\rrbracket\), write
\[
T_j(e_i)=\sum_{\ell=1}^r c_{j,i,\ell}e_\ell,
\]
where \(c_{j,i,\ell} \geq 0\). It follows that
\[
T(e_i)
 = \sum_{j=1}^m T_j(e_i)
 = \sum_{\ell=1}^r
   \left(\sum_{j=1}^m c_{j,i,\ell}\right)e_\ell.
\]
Since \(e_i \in W_I\) and \(T(W_I) \subseteq W_I\), the vector \(T(e_i)\)
belongs to \(W_I\). Thus, for every \(\ell \notin I\), the coefficient of
\(e_\ell\) in the preceding expansion is zero. In other words,
\[
\sum_{j=1}^m c_{j,i,\ell}=0.
\]
The coefficients \(c_{j,i,\ell}\) are nonnegative. Hence
\(c_{j,i,\ell}=0\) for every \(j \in \llbracket 1,m\rrbracket\) and every
\(\ell \notin I\). Therefore
\[
T_j(e_i)=\sum_{\ell\in I}c_{j,i,\ell}e_\ell \in W_I
\]
for every \(j \in \llbracket 1,m\rrbracket\). Since the vectors
\(\{e_i \colon i \in I\}\) span \(W_I\), each map \(T_j\) leaves \(W_I\)
invariant.
\end{proof}

We now set up the coordinate framework that we will use for the remainder of this section. Write
\[
\atoms^+(S)=\{a_1,\dots,a_r\}.
\]
Since the elements of \(\atoms^+(S)\) are \(\qq\)-linearly independent and \(1\in \atoms^+(S)\), we may relabel them so that \(a_1=1\). Define
\[
V\coloneqq \operatorname{span}_{\qq}\{a_1,\dots,a_r\}\subseteq \rr.
\]
The vector space \(V\) is closed under multiplication. Indeed, for each \(i,j\in\llbracket 1,r\rrbracket\), the product \(a_i a_j\) belongs to \(S\). Since \((S,+)\) is a UFM with additive atoms \(a_1,\dots,a_r\), there are unique coefficients \(c_{ij,k}\in\nn_0\) such that
\[
a_i a_j=\sum_{k=1}^r c_{ij,k}a_k.
\]
Bilinearity now shows that the product of two elements of \(V\) belongs to \(V\). Thus \(V\) is a finite-dimensional \(\qq\)-algebra containing \(1\). Moreover, \(V\subseteq\rr\), so \(V\) has no zero divisors. For each nonzero \(x\in V\), multiplication by \(x\) is an injective \(\qq\)-linear map from \(V\) to itself. Finite dimensionality implies that this map is surjective. Hence \(xy=1\) for some \(y\in V\), and \(V\) is a field.

Every element \(x\in V\) has a unique expression
\(
x=q_1a_1+\cdots+q_ra_r
\) with \(q_1,\dots,q_r\in\qq\). We define the coordinate column of \(x\) by
\[
[x]\coloneqq
\begin{bmatrix}
q_1\\
\vdots\\
q_r
\end{bmatrix}.
\]
The coordinate map is \(\qq\)-linear. In particular, \([x+y]=[x]+[y]\) and \([qx]=q[x]\) for all \(x,y\in V\) and \(q\in\qq\). The additive unique factorization property gives the coordinate description
\[
s\in S
\quad\Longleftrightarrow\quad
[s]\in\nn_0^r.
\]
Indeed, each element of \(S\) has a unique expression
\[
s=n_1a_1+\cdots+n_ra_r
\]
with \(n_1,\dots,n_r\in\nn_0\), and every such sum belongs to \(S\). By~\Cref{lem:multiplicative_monoid_reduced}, the element \(a_1=1\) is the only multiplicative unit of \(S\). Consequently, every multiplicative atom \(p\in\atoms(S)\) satisfies \(p>1\). We now consider the nonnegative coordinate cone
\[
C\coloneqq
\{x\in V\colon [x]\in\qq_{\geq 0}^r\}
=
\left\{
\sum_{i=1}^r q_i a_i
\colon
q_1,\dots,q_r\in\qq_{\geq 0}
\right\}.
\]
The cone \(C\) contains \(S\) and is closed under addition and multiplication by elements of \(\qq_{\geq 0}\). It is also closed under multiplication. To see this, take
\[
x=\sum_{i=1}^r q_i a_i
\qquad\text{and}\qquad
y=\sum_{j=1}^r q_j'a_j
\]
in \(C\). Using the structure constants introduced above gives
\[
xy
=
\sum_{i=1}^r\sum_{j=1}^r q_iq_j'a_i a_j
=
\sum_{k=1}^r
\left(
\sum_{i=1}^r\sum_{j=1}^r q_iq_j'c_{ij,k}
\right)a_k.
\]
Each coefficient in the final expression is a nonnegative rational number. Therefore \(xy\in C\). For each nonzero \(v\in C\), the \emph{ray generated by \(v\)} is
\[
\qq_{\geq 0}v
\coloneqq
\{qv\colon q\in\qq_{\geq 0}\}.
\]
We call \(\qq_{\geq 0}v\) an \emph{extreme ray} of \(C\) if every decomposition \(v=x+y\) with \(x,y\in C\) satisfies
\(
x,y\in\qq_{\geq 0}v.
\)
This condition depends only on the ray and not on its chosen generator. Indeed, if \(v'=cv\) for some \(c\in\qq_{>0}\), then a decomposition \(v'=x+y\) yields
\[
v=(1/c)x+(1/c)y.
\]
Thus extremality for the ray generated by \(v\) implies extremality for the ray generated by \(v'\).

Fix \(s\in S\), and let
\[
L_s\colon V\longrightarrow V,
\qquad
x\longmapsto sx,
\]
be the \(\qq\)-linear map given by multiplication by \(s\). Let \(M_s\) denote the matrix of \(L_s\) with respect to the ordered basis \(a_1,\dots,a_r\). Equivalently,
\[
[sx]=M_s[x]
\]
for every \(x\in V\). The \(j\)-th column of \(M_s\) is \([sa_j]\). Since \(sa_j\in S\), this column belongs to \(\nn_0^r\). Hence every entry of \(M_s\) is a nonnegative integer. Finally, set
\[
\sigma\coloneqq a_1+\cdots+a_r\in C.
\]
The coordinate column of \(\sigma\) is the all-ones column, so each of its coordinates is positive. This element will serve as a distinguished point of \(C\) in the arguments below.

The next lemma records the basic algebraic properties of the matrices \(M_s\), together with the irreducibility statement needed for the Perron--Frobenius argument.

\begin{lemma}\label{lem:M_s_properties}
With the notation above, the matrices $M_s$ satisfy $$M_1 = I, \text{\hspace{.3cm}} M_{s+t} = M_s + M_t,\hspace{.25 cm} and \hspace{.3cm} M_{st} = M_s M_t,$$ so that $M_{e^N} = (M_e)^N$ for every $e \in S$ and $N \in \nn$. Moreover, if $e = c_1 a_1 + c_2 a_2 + \cdots + c_r a_r \in S$ with each $c_i \in \nn$ then $M_e$ is irreducible.
\end{lemma}

\begin{proof}
The identity $M_1 = I$ follows from $1 \cdot a_j = a_j$. Additivity follows from $(s+t)a_j = sa_j + ta_j$ and multiplicativity from $(st)a_j = s(ta_j)$ together with linearity of $[\,\cdot\,]$. For irreducibility, write $M_e = c_1 M_{a_1} + c_2 M_{a_2} + \cdots + c_r M_{a_r}$ and suppose for the sake of contradiction that there is a nonempty proper subset $I \subsetneq \{1, \ldots, r\}$ with $M_e(W_I) \subseteq W_I$, where $W_I := \operatorname{span}_\qq \{a_i : i \in I\}$. By~\Cref{lem:each_commp_stabilizes}, each $M_{a_i}$ also satisfies $M_{a_i}(W_I) \subseteq W_I$, which means that $a_i W_I \subseteq W_I$ for every $i$. Since the atoms $a_1, a_2, \ldots, a_r$ span $V$ over $\qq$, it follows that $V W_I \subseteq W_I$. Then $W_I$ is a nonzero proper ideal of the field $V$, a contradiction.
\end{proof}

Now we use Perron--Frobenius (\Cref{thm:PF}) to show that an arbitrary positive element of $V$ can be ``pushed into'' $S$ by multiplying by a high power of an interior element of $C$.

\begin{lemma}\label{lem:positivity}
    Let $e = c_1 a_1 + c_2 a_2 + \cdots + c_ra_r \in S$ in which every coefficient $c_i$ is a positive integer. If $y \in V$ and $y > 0$ then there exist positive integers $D$ and $N$ such that $De^N y \in S$.  
\end{lemma}

\begin{proof}
    Let $M := M_e$. Since $e = \sum_i c_i a_i$ and the matrix of a sum is the sum of the matrices, $M = c_1 M_{a_1} + c_2 M_{a_2} + \cdots + c_r M_{a_r}$. Every $M_{a_i}$ has nonnegative integer entries and every $c_i$ is a positive integer so $M$ has nonnegative integer entries as well. By~\Cref{lem:M_s_properties}, $M$ is irreducible. 
    
    We now prove that every diagonal entry of $M$ is positive. The $(j,j)$ entry of $M$ is the coefficient of $a_j$ in $ea_j$. We know that $ea_j = \sum_{i=1}^r c_i a_i a_j = c_1 a_j + \sum_{i=2}^r c_i a_i a_j$, where $a_1 = 1$ is the first term. Each remaining term $c_i a_i a_j$ is in $S$ and thus contributes nonnegative integers to all coordinates. Therefore the coefficient of $a_j$ in $ea_j$ is at least $c_1$, which is positive. Therefore every diagonal entry of $M$ is positive. Since $M$ is irreducible and has positive trace, it is primitive by~\cite{HJ13}. A matrix is called \emph{primitive} when some power of it has all entries strictly positive. 

    Let $\lambda = (a_1, a_2, \ldots, a_r)$ be the row of the actual atom values (positive real numbers). Note that for each $x \in V$, $\lambda[x] = x$. Replacing $x$ with $ex$ yields $\lambda[ex] = ex$. However, note that $[ex] = M[x]$, so $\lambda[ex] = \lambda M[x]$. Thus $\lambda M[x] = ex = e(\lambda[x])$. Taking $x = a_j$ yields $[a_j] = e_j$, the $j$-th standard basis column. The identity $\lambda M[x] = e \lambda[x]$ then gives $(\lambda M)_j = \lambda M e_j = e \lambda e_j = (e \lambda)_j$ for each $j$. Therefore $\lambda M = e \lambda$ for every entry. 

    For a primitive matrix $M$ with nonnegative entries, the Perron--Frobenius theorem states that there exists a positive real number $\rho$ (greater in magnitude than all other eigenvalues) along with a strictly positive right eigenvector $w$, so $Mw = \rho w$. Our row $\lambda$ has all entries positive and satisfies $\lambda M = e \lambda$, so $e$ must be this distinguished eigenvalue, $\rho$. Furthermore, the theorem states that $e^{-N} M^N \to w \lambda/\lambda w$, a matrix with all positive entries. Let $v = [y]$ be the coordinate column of our given $y$. Since $y > 0$ and $\lambda v = y$, the number $\lambda v$ is positive. Multiplying the limit above by the fixed column $v$ yields $e^{-N} M^N v \to \frac{w(\lambda v)}{\lambda w} = \frac y{\lambda w} w$. Every entry of the right-hand side is positive, since $y > 0$, $\lambda w > 0$, and $w$ has positive entries. 

    A sequence of vectors converging to a vector with all positive entries must eventually have entries that are all positive. So, there exists a positive integer $N$ such that every entry of $e^{-N} M^N v$ is positive and hence (multiplying by the positive number $e^N$) every entry of $M^N v$ is positive. But $M^N v = (M_e)^N [y] = [e^N y]$, which is the coordinate column of $e^N y$. Note that its entries are rational, as they are obtained from the integer matrix $M$ and the rational column $v$. Therefore $e^N y$ has rational coordinates that are all positive. Finally, we choose an integer $D$ that is a common denominator of these coordinates, forcing every coordinate of $D e^N y$ to be a positive integer. Therefore $D e^N y \in S$.
\end{proof}

We now use our previous lemma to show that the reciprocal of each multiplicative atom lies in $C$, and consequently that $C$ is closed under taking multiplicative inverses. 

\begin{lemma}\label{lem:reciprocal}
    If $p$ is a multiplicative atom of $S$ then $\frac{1}{p} \in C$.
\end{lemma}

\begin{proof}
    We know that $p > 1$. Since $S^\ast$ is a UFM, the element $p$ is prime. For each $n \in \nn_0$, the element $\sigma + n = (n+1)a_1 + a_2 + \cdots + a_r$ lies in $S$ and has coordinates $(n+1, 1, 1, \cdots, 1)$, each of which are positive integers. Suppose for the sake of contradiction that $p$ divides $\sigma + n$ for all $n \ge 0$. Then for each $n$, there is an element $q_n \in S$ such that $\sigma + n = p q_n$, or $q_n = (\sigma + n)/p$. Subtracting consecutive terms yields $q_{n+1} - q_n = ((\sigma + n + 1) - (\sigma + n))/p = 1/p$. Let $m_n = [q_n]$ which is in $\nn_0^r$ since $q_n \in S$. Additionally, write $z = [1/p]$, which has integer entries since it is the difference $m_{n+1} - m_n$ of two integer columns. Then $m_{n+1} - m_n = z$ for every $n$, so adding these up yields $m_n = m_0 + nz$ for all $n \ge 0$. If some entry of $z$ were negative then for a large enough $n$, the corresponding entry of $m_0 + nz$ would be negative, contradicting $m_n \in \nn_0^r$. Thus every entry of $z$ is a nonnegative integer, which means that $1/p = z_1 a_1 + z_2 a_2 + \cdots + z_r a_r$ has nonnegative integer coordinates, and so $1/p \in S$. However, $p > 1$ yields $0 < 1/p < 1$, contradicting that there are no elements of $S$ between 0 and 1. Therefore there exists some $n$ such that $p \nmid \sigma + n$.

    Fix such an $n$ and let $e = \sigma + n$. Each of its coordinates are positive so~\Cref{lem:positivity} applies with $y = 1/p > 0$: there exist positive integers $D, N$ with $D e^N \cdot (1/p) \in S$. Therefore $De^N = p \cdot (De^N/p)$ so $p \mid_{S^\ast} De^N$. Since $p$ is prime and does not divide $e$ (and thus does not divide $e^N$), $p$ must divide $D$. Thus, $D/p \in S$. Note that $1/p = 1/D \cdot D/p$ where $D/p \in S \subseteq C$ and $1/D$ is a nonnegative rational. Since $C$ is closed under multiplying by nonnegative rationals, we know that $1/p \in C$.
\end{proof}

Next we use closure under inversion to show that $\sigma$ acts on $C$ as a bijection that carries extreme rays to extreme rays, and we derive a contradiction from the existence of two extreme rays.

\begin{lemma}\label{lem:reciprocal_better}
    If $x \in C$ and $x \neq 0$ then $\frac{1}{x} \in C$.
\end{lemma}

\begin{proof}
    By definition the coordinates of \(x\) are nonnegative rational numbers. Choose a common denominator \(m \in \nn\) for these coordinates. The element \(mx\) has nonnegative integer coordinates and is nonzero. Thus \(mx \in S^\ast\). There exist \(k \in \nn_0\) and multiplicative atoms \(p_1,\dots,p_k\) of \(S\) such that
    \[
         mx=p_1\cdots p_k.
    \]
    Here \(k=0\) precisely when \(mx=1\). In this case the product is interpreted as the empty product \(1\). Lemma~\ref{lem:reciprocal} ensures that \(1/p_i \in C\) for every \(i \in \llbracket 1,k\rrbracket\). Since \(C\) is closed under multiplication, it follows that
    \[
        \frac{1}{mx}=\prod_{i=1}^k \frac{1}{p_i}\in C.
    \]
    Finally \(m \in \nn \subseteq C\), and therefore
    \[
        \frac{1}{x}=m\frac{1}{mx}\in C.
    \]
\end{proof}

We are now in a position to prove the main result of this section.

\begin{theorem}\label{thm:multigenerator}
    If $S := \nn_0[\alpha_1, \ldots, \alpha_n] \subseteq \rr_{\ge 0}$ is a bi-UFS then $S = \nn_0$.
\end{theorem}

\begin{proof}
If $r = 1$ then the only additive atom is $a_1 = 1$, and therefore $S$ just contains nonnegative multiples of $1$, that is, $S = \nn_0$. Thus, for the sake of contradiction, suppose $r \ge 2$. We first show that the ray $\qq_{\ge 0} a_1$ is an extreme ray of $C$. Suppose $a_1 = x + y$ with $x,y \in C$, and let $x = \sum_i p_i a_i$ and $y = \sum_i q_i a_i$ with $p_i, q_i \in \qq_{\ge 0}$. Since the coordinate representation of $a_1$ must be unique, comparing the coordinate of each atom on both sides of the equation yields $1 = p_1 + q_1$ and $p_i + q_i = 0$ for $i \ge 2$. Since $p_i, q_i \ge 0$, the equation $p_i + q_i = 0$ forces $p_i = q_i = 0$ for $i\ge 2$. Thus $x = p_1 a_1$ and $y = q_1 a_1$, so $x,y \in \qq_{\ge 0}a_1$. Thus $\qq_{\ge 0} a_1$ is an extreme ray. By~\Cref{lem:reciprocal_better}, the element $1/\sigma$ is in $C$. We show that $\sigma C = C$. The inclusion $\sigma C \subseteq C$ is immediate from the closure under multiplication. For the reverse inclusion, given $x \in C$, set $y = x/\sigma$. Then $y \in C$ since $1/\sigma \in C$, and $\sigma y = x$, so $x \in \sigma C$. Thus, $\sigma C = C$.  

We next show that multiplication by $\sigma$ carries extreme rays to extreme rays. Let $\qq_{\ge 0} v$ be an extreme ray of $C$. We aim to show that $\qq_{\ge 0} (\sigma v)$ is an extreme ray as well. Note that $\sigma v \neq 0$ since $\sigma \neq 0$ and $v \neq 0$ in the field $V$, so $\qq_{\ge 0} (\sigma v)$ is a ray. Now suppose that $\sigma v = x + y$ with $x,y \in C$. Multiplying both sides by $1/\sigma$ yields $v = (1/\sigma) x + (1/\sigma) y$. Each of $(1/\sigma)x, (1/\sigma)y$ lies in $C$ by closure once again. Since $\qq_{\ge0}v$ is an extreme ray, both $(1/\sigma)x$ and $(1/\sigma)y$ are nonnegative rational multiples of $v$. Therefore let $(1/\sigma) x = pv$ and $(1/\sigma)y = qv$ for some $p, q \in \qq_{\ge 0}$. Multiplying by $\sigma$ yields $x = p(\sigma v)$ and $y = q(\sigma v)$, and therefore $x,y \in \qq_{\ge 0} (\sigma v)$. Hence $\qq_{\ge 0} (\sigma v)$ is an extreme ray.

Since $a_1 = 1$, the equality $\sigma \cdot a_1 = \sigma$ holds, so multiplying by $\sigma$ carries the extreme ray $\qq_{\ge 0} a_1$ to the extreme ray $\qq_{\ge 0} \sigma$ (by the above). We now show that $\qq_{\ge 0}\sigma$ cannot be an extreme ray when $r \ge 2$. Consider the splitting $\sigma = a_1 + (a_2 + \cdots + a_r)$. Since $\qq_{\ge 0}\sigma$ is an extreme ray, $a_1 = q \sigma$ for some $q \in \qq_{\ge 0}$. However this is impossible, since comparing coordinates for $a_2$ on both sides yields $q = 0$, while comparing coordinates for $a_1$ gives $q = 1$, a contradiction. Therefore $r \ge 2$ is impossible, so $r = 1$, in which case the previous argument shows that $S = \nn_0$.
\end{proof}

\bigskip

\section{Reduction to the Complex Multigenerator Case}\label{sec:complex}

In this section, we show that the bi-UFS problem for subsemidomains of $\cc$ with finitely many additive atoms reduces to the positive case treated in the previous sections. The strategy is to use a Perron--Frobenius argument on the matrix of multiplication by the sum of additive atoms: a strictly positive eigenvector $w$ assigns to each additive atom a positive real eigenvalue, and these eigenvalues define a semidomain isomorphism onto a subsemidomain of $\rr_{\ge 0}$.

\begin{theorem}\label{thm:reduction}
    Let $S \subseteq \cc$ be a bi-UFS that is not a UFD. If $|\atoms^+(S)| < \infty$ then $S$ is isomorphic to a positive semidomain.
\end{theorem}

\begin{proof}
    Enumerate $\atoms^+(S)$ as the set $\{ a_1, a_2, \ldots, a_r\}$, where $a_1 = 1$, since we know that $1 \in \atoms^+(S)$. Define the vector space $V := \operatorname{span}_{\qq} \{ a_1, a_2, \ldots, a_r\} \subseteq \cc$. Note that $\{ a_1, a_2, \ldots, a_r\}$ is indeed a $\qq$-basis for $V$ by the uniqueness of additive factorization. That is, $\{ a_1, a_2, \ldots, a_r\}$ are linearly independent over $\qq$. We now show that $V$ is closed under multiplication. Take $x = \sum_i q_i a_i$ and $y = \sum_j p_j a_j$ in $V$. For each pair $i, j$, the product $a_i a_j$ lies in $S$, and hence has a unique additive factorization $a_i a_j = \sum _k c_{ij, k} a_k$, with $c_{ij,k} \in \nn_0$. Therefore \[xy = \sum_{i,j} q_i p_j a_i a_j = \sum_k \left( \sum_{i,j} q_i p_j c_{ij,k}\right)a_k \in V,\] and $V$ is closed under multiplication. Additionally $1 = a_1 \in V$, and $V \subseteq \cc$ has no zero divisors, so $V$ is a finite-dimensional domain over $\qq$. In particular, this implies $V$ is a field.

    For each $s \in S$, define a matrix $M_s \in M_r(\nn_0)$ by multiplication on the additive-atom basis \[sa_j = \sum_{i=1}^r (M_s)_{ij} a_i.\]
    The entries are nonnegative integers since the right-hand side is the additive factorization of $sa_j$. By the same argument as in~\Cref{lem:M_s_properties}, these matrices satisfy $M_1 = I$, $M_{s+t} = M_s + M_t$, and $M_{st} = M_s M_t$. Now define $B := \sum_{i=1}^r M_{a_i} = M_\sigma$, where $\sigma := a_1 + a_2 + \cdots + a_r \in S$. By the irreducibility part of~\Cref{lem:M_s_properties} applied with $e = \sigma$ (whose coefficients are all $1$), $B$ is irreducible.

    By the Perron--Frobenius Theorem, $B$ has a strictly positive eigenvector $w \in \rr_{>0}^r$ with eigenvalue $\rho(B) > 0$. Furthermore, the Perron eigenspace is one dimensional. For every atom $a_i$, the matrices $M_{a_i}$ and $B$ must commute, because multiplication in $S$ is commutative. Therefore $B(M_{a_i} w) = M_{a_i}(Bw) = \rho(B)M_{a_i}w$, and $M_{a_i}w$ is in the Perron eigenspace of $B$. Since the eigenspace is one dimensional, there exists $\lambda_i \in \rr$ such that $M_{a_i} w =\lambda_i w$. We may conclude that $\lambda_i > 0$ because $M_{a_i}$ is nonzero and has nonnegative entries, and $w$ has positive entries. Now define for $s = n_1 a_1 + n_2 a_2 + \cdots + n_r a_r \in S$, the positive realization $s_r := n_1 \lambda_1 + n_2 \lambda_2 + \cdots + n_r \lambda_r \in \rr_{\ge 0}$. Note that this is well-defined because $(S,+)$ is a UFM. Let $R := \{ s_r : s \in S\} \subseteq \rr_{\ge 0}$. We now verify that the map $\phi : S \to R$ defined by $\phi(s) = s_r$ preserves multiplication and addition.

    Addition follows from the definition, so $\phi(s+t) = \phi(s) + \phi(t)$. For multiplication, let $s = \sum_i n_i a_i$ be an element in $S$. Then $M_s = \sum_i n_i M_{a_i}$, so $M_s w = \sum_i n_i M_{a_i} w = \sum_i n_i \lambda_iw = \phi(s)w$. Applying the same identity to $t \in S$ gives $M_t w = \phi(t) w$. Then $M_{st} w = M_s M_t w = M_s (\phi(t) w) = \phi(t) M_s w = \phi(t)\phi(s) w$. On the other hand, $M_{st} w = \phi(st) w$, so $\phi(st) = \phi(s)\phi(t)$. Therefore multiplication is preserved. Note that $\phi(1) = 1$ because $M_1 = I$, and similarly $\phi(0) = 0$.

    It remains to show that $\phi$ is injective. To do this, extend $\phi$ from $S$ to $V$ by defining $\widetilde \phi \left( \sum_i q_i a_i\right) = \sum_i q_i \lambda_i$, for $q_i \in \qq$. This is well-defined since the atoms $a_1, a_2, \ldots, a_r$ are linearly independent over $\qq$. The same matrix computation above shows $\widetilde \phi (xy) = \widetilde \phi(x) \widetilde \phi(y)$ for all $x,y \in V$, and $\widetilde \phi(1) = 1$. If $\widetilde \phi(x) = 0$ for some nonzero $x \in V$ then since $V$ is a field, $x^{-1} \in V$, and so $1 = \widetilde \phi(1) = \widetilde \phi(x x^{-1}) = \widetilde \phi(x) \widetilde \phi(x^{-1}) = 0$, a contradiction. Hence $\widetilde \phi$ is injective. Therefore its restriction $\phi : S \to R$ is injective.

    By the definition of $R$, the map $\phi : S\to R$ is also surjective. Hence $\phi$ is a bijection preserving addition, multiplication, $0$, and $1$. Therefore $S$ is isomorphic to $R$. Since $R \subseteq \rr_{\ge 0}$, we know that $R$ is a positive semidomain. Since semidomain isomorphisms preserve additive and multiplicative factorization, $R$ is bi-UFS whenever $S$ is bi-UFS.
\end{proof}

To apply~\Cref{thm:reduction} to a finitely generated complex semidomain, we need the finiteness hypothesis on additive atoms. The argument from~\Cref{lemma:finitely_atoms} carries over to the complex setting without change.

\begin{lemma}\label{lem:finite_atoms}
    Let $\alpha_1, \alpha_2, \ldots, \alpha_n \in \cc$ be algebraic, and let $S = \nn_0[\alpha_1, \alpha_2, \ldots, \alpha_n] \subseteq \cc$. If $S$ is a bi-UFS that is not a UFD then $|\atoms^+(S)| < \infty$.
\end{lemma}

\begin{proof}
The proof is identical to that of~\Cref{lemma:finitely_atoms}. This is because positivity of the $\alpha_i$ is never used, only that $K := \qq(\alpha_1, \ldots, \alpha_n)$ is a finite extension of $\qq$ containing $S$, and that $(S,+)$ is a UFM. The same argument shows that $\atoms^+(S)$ is $\qq$-linearly independent in $K$, hence finite, with $|\atoms^+(S)| \le [K : \qq] < \infty$.
\end{proof}

Specializing further to the cyclic case, we identify the image $R$ from~\Cref{thm:reduction} as $\nn_0[\beta]$ for an explicit positive real $\beta$.

\begin{corollary}\label{cor:reduction_again}
    Let $\alpha \in \cc$ be algebraic, and let $S = \nn_0[\alpha] \subseteq \cc$. If $S$ is a bi-UFS that is not a UFD then $S$ is isomorphic to $\nn_0[\beta]$ for some $\beta \in \rr_{>0}$.
\end{corollary}

\begin{proof}
    By~\Cref{lem:finite_atoms}, the semidomain $S$ has finitely many additive atoms. Therefore we may apply~\Cref{thm:reduction}, yielding the isomorphism $\phi : S \to R \subseteq \rr_{\ge 0}$. Since $\alpha \in S$, we may set $\beta = \phi(\alpha) \in R$. We claim that $R = \nn_0[\beta]$. Note that every element of $S$ has the form $f(\alpha)$ for some polynomial $f(x) \in \nn_0[x]$. Take an element $r \in R$. Since $R = \phi(S)$, there exists some $s \in S$ such that $r = \phi(s)$. But $s \in S = \nn_0[\alpha]$, so $s = f(\alpha)$ for some $f(x) \in \nn_0[x]$. Therefore $r = \phi(s) =\phi(f(\alpha)) = f(\phi(\alpha)) = f(\beta) \in \nn_0[\beta]$, where the third equality follows from $\phi$ preserving addition and multiplication. Hence $R \subseteq \nn_0[\beta]$.

    Conversely, take an element $f(\beta) \in \nn_0[\beta]$ for some $f(x) \in \nn_0[x]$. Since $\beta = \phi(\alpha)$, we know that $f(\beta) = f(\phi(\alpha)) = \phi(f(\alpha))$. Note that $f(\alpha) \in S$, so $\phi(f(\alpha)) \in \phi(S) = R$. Hence $f(\beta) \in R$ and $\nn_0[\beta] \subseteq R$. Combining both inclusions yields $R = \nn_0[\beta]$. Finally, $\beta \in \rr_{\ge 0}$ since $R \subseteq \rr_{\ge 0}$, and $\beta \ne 0$ since $\phi$ sends nonzero elements to nonzero elements (as $\phi$ is injective), so $\beta \in \rr_{>0}$.
\end{proof}

By~\Cref{lem:finite_atoms} and~\Cref{thm:reduction}, every non-positive bi-UFS of the form $\nn_0[\alpha_1, \ldots, \alpha_n] \subseteq \cc$ with each $\alpha_i$ algebraic is either a UFD or isomorphic to a bi-UFS, and hence is classified by~\Cref{thm:multigenerator}. In particular,~\Cref{cor:reduction_again} yields the cyclic case: every bi-UFS of the form $\nn_0[\alpha] \subseteq \cc$ with $\alpha$ algebraic is either a UFD or isomorphic to $\nn_0[\beta]$ for some $\beta \in \rr_{>0}$, which by~\Cref{thm:final_theorem} forces $\beta \in \nn$.

\bigskip

\section*{Acknowledgments}

The authors are grateful to the CrowdMath Internship (CMI), a year-long mathematics research program hosted by the MIT Department of Mathematics, for making this work possible. They especially thank their CMI research mentors, Victor Gonzalez and Felix Gotti, for their dedicated guidance during the reading period. The authors also thank the PRIMES program for fostering this research experience and providing a supportive academic environment in which to pursue this work. During the preparation of this manuscript, the sixth author was supported by the National Science Foundation through an Ascend Postdoctoral Fellowship under award No.~2513588.



\bigskip

\end{document}